\documentclass[11pt,reqno]{amsart}
\usepackage{mathrsfs}
\usepackage{cite}

\def\subjclass#1{{\renewcommand{\thefootnote}{}%
\footnote{\emph{Mathematics Subject Classification (2020):} #1}}}
\usepackage{amsfonts}
\usepackage{amssymb}

\date{\today}

\hoffset=- 2cm \voffset=- 0cm 
 \theoremstyle{plain}
\newtheorem{Thm}{Theorem}

\newtheorem{Rem}[Thm]{Remark}
\newtheorem{Lem}[Thm]{Lemma}
\newtheorem{Cor}[Thm]{Corollary}

\usepackage{amssymb}
\usepackage{color}

\def\0{\mathbf 0}

\def\v{\vskip}

\errorcontextlines=0 \numberwithin{equation}{section}
\numberwithin{Thm}{section}
\allowdisplaybreaks

\begin{document}
\large

\title[Liouville type theorems]
{New Liouville type theorems for 3D steady incompressible MHD equations and Hall-MHD equations}

\author{Zhibing Zhang}

\address{Zhibing Zhang: School of Mathematics and Physics, Key Laboratory of Modeling, Simulation and Control of Complex Ecosystem in Dabie Mountains of Anhui Higher Education Institutes, Anqing Normal University, Anqing 246133, China}
\email{zhibingzhang29@126.com}%

\thanks{}



\begin{abstract}
In this paper, we study Liouville type results for the three-dimensional stationary incompressible MHD equations and Hall-MHD equations.
Using the energy method and an iteration argument, we establish Liouville type theorems if Lebesgue norms of the velocity and magnetic field on the annulus satisfy certain growth conditions. Furthermore, by establishing new energy estimates and developing some novel differential inequality techniques, we relax the growth conditions by logarithmic factors and obtain logarithmic improvement version of Liouville type theorems. For the MHD equations, the assumptions imposed on the magnetic field are weaker and wider than that of the velocity field in certain sense. Our results extend and improve several recent works.
\end{abstract}
\maketitle

\section{Introduction}
We consider the following two systems in $\mathbb{R}^{3}$: the stationary incompressible MHD equations
\begin{equation}\label{equ1.1}
  \left\{
    \begin{array}{ll}
     -\Delta u+(u\cdot\nabla) u+\nabla \pi=(B\cdot\nabla)B,  \\
   -\Delta B+(u\cdot\nabla) B-(B\cdot\nabla)u=0, \\
     \mathrm{div}u=\mathrm{div} B=0,
\end{array}
  \right.
\end{equation}
and the stationary incompressible Hall-MHD equations
\begin{equation}\label{equ1.2}
  \left\{
    \begin{array}{ll}
     -\Delta u+(u\cdot\nabla) u+\nabla \pi=(B\cdot\nabla)B,  \\
   -\Delta B+\mathrm{curl}(\mathrm{curl} B\times B)+(u\cdot\nabla) B-(B\cdot\nabla)u=0, \\
    \mathrm{div} u=\mathrm{div} B=0,
\end{array}
  \right.
\end{equation}
where the curl operator is defined by $\mathrm{curl}B=\nabla\times B$.
In \eqref{equ1.1} and \eqref{equ1.2}, $u$ represents the velocity vector field, $B$ denotes the magnetic vector field, and $\pi$ is the scalar pressure. Notice that when $B=0$, both \eqref{equ1.1} and \eqref{equ1.2} reduce to the stationary incompressible Navier-Stokes equations
\begin{equation}\label{equ1.3}
  \left\{
    \begin{array}{ll}
     -\Delta u+(u\cdot\nabla) u+\nabla \pi=0,  \\
    \mathrm{div}u=0.
\end{array}
  \right.
\end{equation}

The Liouville type problem for the Navier-Stokes equations is an active research topic in the community of
mathematical fluid mechanics and has been studied extensively, see \cite{BGWX25,CPZ20,CPZZ20,Chae14,CW16,CW19,CJL21,CY25,Galdi,GW78,KNSS09,KTW17,KTW22,KTW24,Seregin16,Seregin18,SW19,Tsai21} and the references therein. In the last decade, there have also been a lot of progress on the Liouville type problem for the three-dimensional stationary incompressible MHD equations and Hall-MHD equations. In 2014, Chae et al. \cite{CDL14} proved that the smooth solutions of \eqref{equ1.2} must be trivial provided that $u,B\in L^\infty(\mathbb{R}^3)\cap L^\frac{9}{2}(\mathbb{R}^3)$ and
$\nabla u,\nabla B\in L^2(\mathbb{R}^3)$. In 2016, Chae and Weng \cite{CWeng16} showed that the smooth solutions of \eqref{equ1.1} and \eqref{equ1.2} are zero if $u,B$ satisfy $\nabla u,\nabla B\in L^2(\mathbb{R}^3)$ and $u\in L^3(\mathbb{R}^3)$. In 2020, Yuan and Xiao \cite{YX20} obtained Liouville type theorems for \eqref{equ1.1} and \eqref{equ1.2} without assuming finite Dirichlet energies. They proved that the smooth solutions of \eqref{equ1.1} are trivial provided $u,B\in L^p(\mathbb{R}^3)$ while the smooth solutions of \eqref{equ1.2} are trivial provided $u\in L^p(\mathbb{R}^3)$, $B\in L^q(\mathbb{R}^3)$, where $p\in[2,\frac{9}{2}]$, $q\in[4,\frac{9}{2}]$. Fan and Wang \cite{FW21} extended Chae-Wolf's logarithmic type integrability condition (see \cite{CW16}) to the MHD equations, and claimed that if the smooth solutions of \eqref{equ1.1} satisfy
\begin{equation*}
\int_{\mathbb{R}^{3}}|(u,B)|^{\frac{9}{2}}\left\{\ln\left(2+\frac{1}{|(u,B)|}\right)\right\}^{-1}dx<+\infty,
\end{equation*}
then $u=B=0$.

On the other hand, Chae et al. \cite{CKW22} showed that the solutions of the MHD equations and Hall-MHD equations are trivial under suitable growth conditions at infinity for the mean oscillations for the potential functions. Assume that there exist smooth matrix-valued functions $\Phi$ and $\Psi$ such that $u=\mathrm{div}\Phi$ and $B=\mathrm{div}\Psi$, and $\Phi,\Psi$
satisfy
\begin{equation*}
\left(\frac{1}{|B_R|}\int_{B_R}|\Phi-\Phi_{B_R}|^sdx\right)^\frac{1}{s}+\left(\frac{1}{|B_R|}\int_{B_R}|\Psi-\Psi_{B_R}|^sdx\right)^\frac{1}{s}\leq CR^{\frac{1}{3}-\frac{1}{s}}, \forall R>1,
\end{equation*}
for some $3<s\leq6$, where $\Phi_{B_R}$ and $\Psi_{B_R}$ are the mean value of $\Phi$ and $\Psi$ on the ball $B_R$, respectively. Then they claimed that the smooth solutions of \eqref{equ1.1} must be trivial. In addition, if the following condition holds
\begin{equation*}
\left(\frac{1}{|B_R|}\int_{B_R}|B-B_{B_R}|^pdx\right)^\frac{1}{p}=o\left(R^{\left(\frac{2s}{3}+1\right)\left(\frac{1}{3}-\frac{1}{p}\right)}\right)\text{ as } R\rightarrow+\infty,
\end{equation*}
for some $p>3$, where $B_{B_R}$ represents the mean value of $B$ on the ball $B_R$, then the smooth solutions of \eqref{equ1.2} must be trivial.
As a direct consequence, Liouville type theorem for \eqref{equ1.1} can be obtained under the assumption $u,B\in BMO^{-1}(\mathbb{R}^3)$, which was firstly stated in \cite[Corollary 1.2]{CW21}.
Chen et al. \cite{CLW22} established Liouville type theorems for \eqref{equ1.1} and \eqref{equ1.2} provided $u\in BMO^{-1}(\mathbb{R}^3)$ and $B\in L^{6,\infty}(\mathbb{R}^3)$.
Moreover, they also established Liouville type theorems for \eqref{equ1.1} under the condition $u\in L^3(\mathbb{R}^3)$ and $B\in L^6(\mathbb{R}^3)$.
Very recently, Chae and Lee \cite{CL24} extended the results of Chen et al. \cite{CLW22}. They proved Liouville type theorems for \eqref{equ1.1} under one of the following assumptions
\begin{subequations}
\begin{align}
&(\mathrm{i})\;u\in M^{-1}_{s,r}(\mathbb{R}^3),\;B\in L^q(\mathbb{R}^3),\;3<s\leq6,\; 0\leq r\leq\frac{s}{3}-1,\;3\leq q\leq\frac{6s}{s+2r},\\
&(\mathrm{ii})\;u\in L^p(\mathbb{R}^3),\;B\in L^q(\mathbb{R}^3),\;\frac{1}{p}+\frac{2}{q}\geq\frac{2}{3},\;3\leq p\leq\frac{9}{2},\;3\leq q\leq6.\label{cond1.6b}
\end{align}
\end{subequations}
For the definition of the space $M^{-1}_{s,r}(\mathbb{R}^3)$, one can refer to \cite[pp.2]{CL24}. Besides, Cho et al. \cite{CNY24} showed Liouville type theorems for \eqref{equ1.1} provided  one of the following assumptions holds
$$
\aligned
&(\mathrm{i})\liminf_{R\rightarrow+\infty}\left(\frac{\|u\|_{L^p(B_{2R}\backslash B_R)}}{R^{\frac{2}{p}-\frac{1}{3}}}+\frac{\|B\|_{L^q(B_{2R}\backslash B_R)}}{R^{\frac{2}{q}-\frac{1}{3}}}\right)<+\infty,\;\frac{3}{2}<p<3,\; 1\leq q\leq 3,\\
&(\mathrm{ii})\liminf_{R\rightarrow+\infty}\frac{\|u\|_{L^3(B_{2R}\backslash B_R)}}{R^\frac{1}{3}}=0,\;\limsup_{R\rightarrow+\infty}\frac{\|B\|_{L^q(B_{2R}\backslash B_R)}}{R^{\frac{2}{q}-\frac{1}{3}}}<+\infty,\;1\leq q\leq 3.
\endaligned
$$
Cho et al. \cite{CNY24} also showed Liouville type theorems for \eqref{equ1.1} and \eqref{equ1.2} under one of the following conditions
\begin{subequations}
\begin{align}
&(\mathrm{i})\liminf_{R\rightarrow+\infty}\frac{\|u\|_{L^p(B_{2R}\backslash B_R)}}{R^{\frac{2}{p}-\frac{1}{3}}}=0,\; \limsup_{R\rightarrow+\infty}\frac{\|B\|_{L^q(B_{2R}\backslash B_R)}}{R^{\frac{2}{q}-\frac{1}{3}}}<+\infty,\;\frac{3}{2}<p<3,\;q=2p',\label{cond1.7a}\\
&(\mathrm{ii})\liminf_{R\rightarrow+\infty}\frac{\|u\|_{L^p(B_{2R}\backslash B_R)}}{R^{\frac{3}{p}-1}}=0,\;\limsup_{R\rightarrow+\infty}\|B\|_{L^6(B_{2R}\backslash B_R)}<+\infty,\;\frac{3}{2}<p\leq3,\label{cond1.7b}
\end{align}
\end{subequations}
where $p'$ denotes the conjugate exponent to $p$ and is given by $p'=\frac{p}{p-1}$.

 In this paper, we are mainly concerned with Liouville type theorems under integrability conditions. From the perspective of the regularity theory for the MHD equations, we see that the velocity field plays a more important role than the magnetic field. For example, He-Xin \cite{HX05} and Wang-Zhang \cite{WZ13} showed some regularity criterions for the MHD equations only in terms of the velocity field.
 Under the guidance of this idea and inspired by the works of \cite{CNY24,CY25}, we succeed in establishing Liouville type theorems for \eqref{equ1.1} under weaker and wider conditions imposed on the magnetic field than the velocity field in certain sense.
Using the energy method and an iteration argument, we obtain Liouville type theorems if Lebesgue norms of the velocity and magnetic field on the annulus satisfy certain growth conditions. Furthermore, combining the energy method, the iteration argument and some differential inequality techniques, we relax the growth conditions by logarithmic factors and obtain logarithmic improvement version of Liouville type theorems.
Our results extend and improve the recent works of Chae-Lee \cite{CDL14}, Cho-Neustupa-Yang \cite{CNY24} and Cho-Yang \cite{CY25}.

For simplicity, we denote
$$X_{p,\alpha}(R)=R^{-\alpha}\|u\|_{L^p\left(A_R\right)},\;Y_{q,\beta}(R)=R^{-\beta}\|B\|_{L^q\left(A_R\right)},$$
$$X_{p,\alpha,\lambda}(R)=R^{-\alpha}(\ln R)^{-\lambda}\|u\|_{L^p\left(A_R\right)},\;Y_{q,\beta,\mu}(R)=R^{-\beta}(\ln R)^{-\mu}\|B\|_{L^q\left(A_R\right)},$$
where $A_R$ represents the annulus $B_{2R}\backslash \overline{B_{\frac{3R}{2}}}$.

Our main results are stated as follows.

\begin{Thm}\label{main1}
Let $(u,\pi,B)$ be a smooth solution of \eqref{equ1.1}. Suppose that one of the following assumptions holds
\begin{align*}
\mathrm{(A1)}\;&\liminf\limits_{R\rightarrow+\infty}\left[X_{p,\alpha}(R)+Y_{q,\beta}(R)\right]<+\infty,\text{ where  $p,q,\alpha,\beta$ satisfy}\\
&p\in\left(\frac{3}{2},3\right),\;q\in[1,2p'),\;\alpha\in\left[0,\frac{2}{p}-\frac{1}{3}\right],\;\beta\in\left[0,\frac{3}{q}-\frac{1}{2}\right],\;\alpha+\frac{(4p-6)q}{(6-q)p}\beta\leq1;\\
\mathrm{(A2)}\;&\liminf\limits_{R\rightarrow+\infty}\left[X_{p,\alpha}(R)+Y_{q,\beta}(R)\right]<+\infty,\text{ where  $p,q,\alpha,\beta$ satisfy}\\
&p\in\left(\frac{3}{2},3\right),\;q\in[2p',6],\;\alpha\in\left[0,\frac{2}{p}-\frac{1}{3}\right],\;\beta\in\left[0,\frac{3}{q}-\frac{1}{2}\right],\;\alpha+2\beta<\frac{3}{p}+\frac{6}{q}-2;\\
\mathrm{(A3)}\;&\liminf\limits_{R\rightarrow+\infty}X_{p,\alpha}(R)=0,\;\limsup\limits_{R\rightarrow+\infty}Y_{q,\beta}(R)<+\infty,\text{ where  $p,q,\alpha,\beta$ satisfy}\\
&p\in\left[3,\frac{9}{2}\right],\;q\in[1,2p'),\;\alpha\in\left[0,\frac{3}{p}-\frac{2}{3}\right],\;\beta\in\left[0,\frac{3}{q}-\frac{1}{2}\right],\;\alpha+\frac{(4p-6)q}{(6-q)p}\beta\leq1;\\
\mathrm{(A4)}\;&\liminf\limits_{R\rightarrow+\infty}X_{p,\alpha}(R)=0,\;\limsup\limits_{R\rightarrow+\infty}Y_{q,\beta}(R)<+\infty,\text{ where  $p,q,\alpha,\beta$ satisfy}\\
&p\in\left[3,\frac{9}{2}\right],\;q\in[2p',6],\;\alpha\in\left[0,\frac{3}{p}-\frac{2}{3}\right],\;\beta\in\left[0,\frac{3}{q}-\frac{1}{2}\right],\;\alpha+2\beta\leq\frac{3}{p}+\frac{6}{q}-2.
\end{align*}
Then $u =B= 0$.
\end{Thm}

\begin{Thm}\label{main2}
Let $(u,\pi,B)$ be a smooth solution of \eqref{equ1.2}. Suppose that one of the following assumptions holds
\begin{align*}
\mathrm{(B1)}\;&\liminf\limits_{R\rightarrow+\infty}\left[X_{p,\alpha}(R)+Y_{q,\beta}(R)\right]<+\infty,\text{ where  $p,q,\alpha,\beta$ satisfy}\\
&p\in\left(\frac{3}{2},3\right),\;q\in(3,2p'),\;\alpha\in\left[0,\frac{2}{p}-\frac{1}{3}\right],\;\beta\in\left[0,\frac{3}{q}-\frac{1}{2}\right],\;\alpha+\frac{(4p-6)q}{(6-q)p}\beta\leq1;\\
\mathrm{(B2)}\;&\liminf\limits_{R\rightarrow+\infty}\left[X_{p,\alpha}(R)+Y_{q,\beta}(R)\right]<+\infty,\text{ where  $p,q,\alpha,\beta$ satisfy}\\
&p\in\left(\frac{3}{2},3\right),\;q\in[2p',6],\;\alpha\in\left[0,\frac{2}{p}-\frac{1}{3}\right],\;\beta\in\left[0,\frac{3}{q}-\frac{1}{2}\right],\;\alpha+2\beta<\frac{3}{p}+\frac{6}{q}-2;\\
\mathrm{(B3)}\;&\liminf\limits_{R\rightarrow+\infty}X_{p,\alpha}(R)=0,\;\limsup\limits_{R\rightarrow+\infty}Y_{q,\beta}(R)<+\infty,\text{ where  $p,q,\alpha,\beta$ satisfy}\\
&p\in\left[3,\frac{9}{2}\right],\;q\in(3,6],\;\alpha\in\left[0,\frac{3}{p}-\frac{2}{3}\right],\;\beta\in\left[0,\frac{3}{q}-\frac{1}{2}\right],\;\alpha+2\beta\leq\frac{3}{p}+\frac{6}{q}-2.
\end{align*}
Then $u =B= 0$.
\end{Thm}

We notice that the inequality $\alpha+2\beta<\frac{3}{p}+\frac{6}{q}-2$ in (A2) and (B2) is a strict inequality, which means the conclusion will fail when $(\alpha,\beta)$ is on the line $\alpha+2\beta=\frac{3}{p}+\frac{6}{q}-2$. However, we find that the result still holds for some special endpoint case.
\begin{Rem}\label{Rem1.3}
Both $\mathrm{(A2)}$ and $\mathrm{(B2)}$ can be replaced by the following assumption
\begin{align*}
\mathrm{(A5)}\;&\liminf\limits_{R\rightarrow+\infty}\left[X_{p,\alpha}(R)+Y_{q,\beta}(R)\right]<+\infty,\text{ where  $p,q,\alpha,\beta$ satisfy}\\
&p\in\left(\frac{3}{2},3\right),\;q\in[2p',6],\;\alpha=\frac{3}{p}-1,\;\beta=\frac{3}{q}-\frac{1}{2}.
\end{align*}
\end{Rem}

Roughly speaking, we can relax the growth conditions in Theorem \ref{main1} and Theorem \ref{main2} by logarithmic factors in certain situations. For the convenience of presenting our Liouville type results of logarithmic improvement version, we make two basic assumptions on the parameters $\lambda$ and $\mu$. The first one is
\begin{align}\label{ass1.8}
\lambda\in\left[0,\frac{3}{p}-1\right] \text{ for }p\in\left(\frac{3}{2},3\right),\;\mu\geq0.
\end{align}
The second one is
\begin{align}\label{ass1.9}
\text{When }&p\in\left(\frac{3}{2},3\right),\;q\in[1,2p'),\;\beta\in\left[0,\frac{3}{q}-\frac{1}{2}\right),\text{ we assume}\notag\\
&\lambda+\frac{(4p-6)q}{(6-q)p}\mu\leq\frac{6p-3q(p-1)}{(6-q)p};\notag\\
\text{when }&p\in\left(\frac{3}{2},3\right),\;q\in[1,2p'),\;\beta=\frac{3}{q}-\frac{1}{2},\text{ we assume}\notag\\
&\lambda+\frac{6p+pq-3q}{(6-q)p}\mu\leq\frac{6p-3q(p-1)}{2(6-q)p};\\
\text{when }&p\in\left[3,\frac{9}{2}\right],\;q\in[1,2p'),\;\beta\in\left[0,\frac{3}{q}-\frac{1}{2}\right),\text{ we assume }\notag\\ &\mu\in\left[0,\frac{6p-3q(p-1)}{(4p-6)q}\right];\notag\\
\text{when }&p\in\left[3,\frac{9}{2}\right],\;q\in[1,2p'),\;\beta=\frac{3}{q}-\frac{1}{2},\text{ we assume }\notag\\
&\mu\in\left[0,\frac{6p-3q(p-1)}{2(6p+pq-3q)}\right].\notag
\end{align}
Now we are ready to show our Liouville type results of logarithmic improvement version.
\begin{Thm}\label{main3}
Let $(u,\pi,B)$ be a smooth solution of \eqref{equ1.1}. Assume that $\lambda,\mu$ satisfy \eqref{ass1.8} and \eqref{ass1.9}.  Suppose that one of the following assumptions holds
\begin{align*}
\mathrm{(C1)}\;&\limsup\limits_{R\rightarrow+\infty}\left[X_{p,\alpha,\lambda}(R)+Y_{q,\beta,\mu}(R)\right]<+\infty,\text{ where  $p,q,\alpha,\beta$ satisfy}\\
&p\in\left(\frac{3}{2},3\right),\;q\in[1,2p'),\;\alpha\in\left[0,\frac{2}{p}-\frac{1}{3}\right],\;\beta\in\left[0,\frac{3}{q}-\frac{1}{2}\right],
\;\alpha+\frac{(4p-6)q}{(6-q)p}\beta\leq1,\\
\mathrm{(C2)}\;&\limsup\limits_{R\rightarrow+\infty}\left[X_{p,\alpha,\lambda}(R)+Y_{q,\beta,\mu}(R)\right]<+\infty,\text{ where  $p,q,\alpha,\beta$ satisfy}\\
&p\in\left(\frac{3}{2},3\right),\;q\in[2p',6],\;\alpha\in\left[0,\frac{2}{p}-\frac{1}{3}\right],\;\beta\in\left[0,\frac{3}{q}-\frac{1}{2}\right],\;
\alpha+2\beta<\frac{3}{p}+\frac{6}{q}-2;\\
\mathrm{(C3)}\;&\lim\limits_{R\rightarrow+\infty}X_{p,\alpha}(R)=0,\;\limsup\limits_{R\rightarrow+\infty}Y_{q,\beta,\mu}(R)<+\infty,\text{ where  $p,q,\alpha,\beta$ satisfy}\\
&p\in\left[3,\frac{9}{2}\right],\;q\in[1,2p'),\;\alpha\in\left[0,\frac{3}{p}-\frac{2}{3}\right],\;\beta\in\left[0,\frac{3}{q}-\frac{1}{2}\right],
\;\alpha+\frac{(4p-6)q}{(6-q)p}\beta\leq1.
\end{align*}
Then $u =B= 0$.
\end{Thm}

\begin{Thm}\label{main4}
Let $(u,\pi,B)$ be a smooth solution of \eqref{equ1.2}. Assume that $\lambda,\mu$ satisfy \eqref{ass1.8} and \eqref{ass1.9}. Suppose that one of the following assumptions holds
\begin{align*}
\mathrm{(D1)}\;&\limsup\limits_{R\rightarrow+\infty}\left[X_{p,\alpha,\lambda}(R)+Y_{q,\beta,\mu}(R)\right]<+\infty,\text{ where  $p,q,\alpha,\beta$ satisfy}\\
&p\in\left(\frac{3}{2},3\right),\;q\in(3,2p'),\;\alpha\in\left[0,\frac{2}{p}-\frac{1}{3}\right],\;\beta\in\left[0,\frac{3}{q}-\frac{1}{2}\right],\;
\alpha+\frac{(4p-6)q}{(6-q)p}\beta\leq1;\\
\mathrm{(D2)}\;&\limsup\limits_{R\rightarrow+\infty}\left[X_{p,\alpha,\lambda}(R)+Y_{q,\beta,\mu}(R)\right]<+\infty,\text{ where  $p,q,\alpha,\beta$ satisfy}\\
&p\in\left(\frac{3}{2},3\right),\;q\in[2p',6],\;\alpha\in\left[0,\frac{2}{p}-\frac{1}{3}\right],\;\beta\in\left[0,\frac{3}{q}-\frac{1}{2}\right],\;
\alpha+2\beta<\frac{3}{p}+\frac{6}{q}-2.
\end{align*}
Then $u =B= 0$.
\end{Thm}

For the special endpoint case $(\alpha,\beta)=(\frac{3}{p}-1,\frac{3}{q}-\frac{1}{2})$, we also have the logarithmic improvement version.
\begin{Rem}\label{Rem1.6}
Both $\mathrm{(C2)}$ and $\mathrm{(D2)}$ can be replaced by the following assumption
\begin{align*}
\mathrm{(C4)}\;&\limsup\limits_{R\rightarrow+\infty}\left[X_{p,\alpha,\lambda}(R)+Y_{q,\beta,\mu}(R)\right]<+\infty,\text{ where  $p,q,\alpha,\beta,\lambda,\mu$ satisfy}\\
p\in&\left(\frac{3}{2},3\right),\;q\in[2p',6],\;\alpha=\frac{3}{p}-1,\;\beta=\frac{3}{q}-\frac{1}{2},\;\lambda\in\left[0,\frac{3}{p}-1\right],
\;\mu\geq0,\;\lambda+\mu\leq\frac{1}{2}.
\end{align*}
\end{Rem}

As a consequence of Theorem \ref{main1} and Theorem \ref{main2}, we can show that the velocity and magnetic field are trivial provided that they belong to some Lebesgue spaces or satisfy some decay conditions at infinity.

\begin{Cor}\label{Cor1.7}
Let $(u,\pi,B)$ be a smooth solution of \eqref{equ1.1}. If $u\in L^p(\mathbb{R}^3)$, $B\in L^q(\mathbb{R}^3)$, and $p,q$ satisfy
$$p\in \left(\frac{3}{2},\frac{9}{2}\right],\;q\in[1,6],\;\frac{1}{p}+\frac{2}{q}\geq\frac{2}{3},$$
then $u = B=0$.
\end{Cor}

\begin{Cor}\label{Cor1.8}
Let $(u,\pi,B)$ be a smooth solution of \eqref{equ1.2}. If $u\in L^p(\mathbb{R}^3)$, $B\in L^q(\mathbb{R}^3)$, and $p,q$ satisfy
$$p\in \left(\frac{3}{2},\frac{9}{2}\right],\;q\in(3,6],\;\frac{1}{p}+\frac{2}{q}\geq\frac{2}{3},$$
then $u = B=0$.
\end{Cor}
Since the condition $u(x)=o(|x|^{-\frac{2}{3}})$, $B(x)=O(|x|^{-\frac{2}{3}})$ as $|x|\rightarrow+\infty$ indicates
$$
\lim_{R\rightarrow+\infty}\|u\|_{L^\frac{9}{2}\left(A_R\right)}=0,\;\limsup_{R\rightarrow+\infty}\|B\|_{L^\frac{9}{2}\left(A_R\right)}<+\infty,
$$
applying Theorem \ref{main1} and Theorem \ref{main2} with $\alpha=\beta=0$, $p=q=\frac{9}{2}$, we have the following corollary.
\begin{Cor}
Let $(u,\pi,B)$ be a smooth solution of \eqref{equ1.1} or \eqref{equ1.2}. If $u(x)=o(|x|^{-\frac{2}{3}})$, $B(x)=O(|x|^{-\frac{2}{3}})$ as $|x|\rightarrow+\infty$, then $u = B=0$.
\end{Cor}

Finally, we give a summary of our work. We use a new iteration procedure to establish Liouville type theorems under the assumption that the velocity and magnetic field satisfy certain growth conditions of Lebesgue norms on the annulus. In our iteration procedure, the $L^2$ norms of $\nabla u$, $\nabla B$ and the $L^6$ norms of $u$, $B$ are iterated together.
We realize that the conditions on the magnetic field can be weaker and wider than the velocity field. We impose some requirements on $R^{-\alpha}\|u\|_{L^p\left(A_R\right)}$ for the velocity field and on $R^{-\beta}\|B\|_{L^q\left(A_R\right)}$ for the magnetic field, where the range of the parameters $\beta$ and $q$ is wider than that of $\alpha$ and $p$. It is not difficult to verify that our results cover and extend all Liouville results for \eqref{equ1.1} and \eqref{equ1.2} in \cite{CNY24}. Moreover, owing to Remark \ref{Rem1.3} and Remark \ref{Rem3.1}, the condition \eqref{cond1.7a} has an alternative condition
$$
\limsup_{R\rightarrow+\infty}\frac{\|u\|_{L^p(B_{2R}\backslash B_R)}}{R^{\frac{2}{p}-\frac{1}{3}}}<+\infty,\;\liminf_{R\rightarrow+\infty}\frac{\|B\|_{L^q(B_{2R}\backslash B_R)}}{R^{\frac{2}{q}-\frac{1}{3}}}=0,\;\frac{3}{2}<p<3,\;q=2p',
$$
while the condition \eqref{cond1.7b} with $p\neq3$ can be weaken to
$$\liminf_{R\rightarrow+\infty}\left(\frac{\|u\|_{L^p(B_{2R}\backslash B_R)}}{R^{\frac{3}{p}-1}}+\|B\|_{L^6(B_{2R}\backslash B_R)}\right)<+\infty,\;\frac{3}{2}<p<3.$$
Obviously, Corollary \ref{Cor1.7} improves the corresponding result of Chae-Lee \cite[Theorem 1.1]{CL24} (also see \eqref{cond1.6b}) by widening the range of $p$ and $q$.

On the other hand, Cho-Yang \cite{CY25} got a logarithmic improvement of Liouville type results for the Navier-Stokes equations. Specifically, they proved the triviality of the smooth solution of \eqref{equ1.3} if
$$
\limsup_{R\rightarrow+\infty}\frac{\|u\|_{L^p(A_R)}}{R^{\frac{2}{p}-\frac{1}{3}}(\ln R)^{\frac{3}{p}-1}}<+\infty,\;\frac{3}{2}<p<3.
$$
They proved this result by a special classified discussion based on
\begin{align*}
\liminf_{R\rightarrow+\infty}\frac{\|u\|_{L^3(A_R)}}{R^{\frac{2}{3}-\frac{1}{p}}(\ln R)^{\frac{3}{p}-1}}<+\infty\text{ or }=+\infty.
\end{align*}
To the best of our knowledge, their method can not be applied to obtain Liouville type results of the logarithmic version for the MHD equations and Hall-MHD equations. By developing a new method, we avoid the classified discussion and succeed in extending Cho-Yang's logarithmic improvement of Liouville type results for the Navier-Stokes equations to the MHD equations and Hall-MHD equations. Obviously, our results of logarithmic version improve the results of Cho-Neustupa-Yang \cite{CNY24} by logarithmic factors.

The rest of this paper is arranged as follows. In Section \ref{sec2}, we introduce the property of the Bogovskii operator and a useful iteration lemma. At the same time, we establish some key lemmas, which play an important role in the proof of Liouville type theorems. Section \ref{sec3} is devoted to proving Theorem \ref{main1} and Theorem \ref{main2} while Section \ref{sec4} is devoted to proving Theorem \ref{main3} and Theorem \ref{main4}. Throughout this article, we use $C$ to denote a finite inessential constant which may be different from line to line.

\v0.1in
\section{Preliminaries}\label{sec2}

The first lemma is the existence and boundedness of the Bogovskii map, which is used to handle the pressure term $\nabla \pi$.
\begin{Lem}\label{Lem2.1}$($See \cite[Lemma 1]{Tsai21}$)$
Let $E$ be a bounded Lipschitz domain in $\mathbb{R}^{3}$. Denote $L_{0}^r(E):=\{f\in L^r(E):\int_E fdx=0\}$ with $1<r<\infty$. There exists a linear operator
\begin{equation*}
\mathrm{Bog}:L_{0}^r(E)\rightarrow W_{0}^{1,r}(E),
\end{equation*}
such that for any  $f\in L_{0}^r(E),w=\mathrm{Bog}f$ is a vector field satisfying
\begin{equation*}
w\in W_{0}^{1,r}(E), \hspace{0.3cm}\mathrm{div}w=f,\hspace{0.3cm}\|\nabla w\|_{L^r(E)}\leq C_{\mathrm{Bog}}(E,r)\|f\|_{L^r(E)},
\end{equation*}
where the constant $C_{\mathrm{Bog}}(E,r)$  is independent of $f$. By a rescaling argument, we see
$$C_{\mathrm{Bog}}(RE,r)=C_{\mathrm{Bog}}(E,r),\text{ where $RE=\{Rx:x\in E\}$.}$$
\end{Lem}

Actually, when $E$ is an annular region, we can explicitly calculate the constant in the above lemma.
\begin{Lem}\label{Lem2.2}$($See \cite[Lemma 3]{Tsai21}$)$ Let $R>0,1<k<\infty$ and $E=B_{kR}\setminus \overline{B_R}$ be an annulus in $\mathbb{R}^{3}$. Then the constant $C_{\mathrm{Bog}}(E,r)$ from Lemma \ref{Lem2.1} can be chosen as $C_{\mathrm{Bog}}(E,r)=\frac{C_r}{(k-1)k^{1-\frac{1}{r}}}$, where $C_r$ is independent of $k$ and $R$.
\end{Lem}
Next, we show the following standard iteration lemma, which is a generalization and a direct consequence of \cite[Lemma 3.1]{Giaquinta}, and can be found in \cite[Lemma 2.1]{CL24} or \cite[Lemma 2.2]{YX20}.
\begin{Lem}\label{Lem2.3}
Let $f(t)$ be a non-negative bounded function on $\left[r_0, r_1\right] \subset \mathbb{R}^{+}$. If there are non-negative constants $A_i, B_i,\alpha_i$, $i=1,2,\cdots,m$, and a parameter $\theta \in[0,1)$ such that for any $r_0 \leq s<t \leq r_1$, it holds that
$$
f(s) \leq \theta f(t)+\sum_{i=1}^m\left(\frac{A_i}{(t-s)^{\alpha_i}}+B_i\right),
$$
then
$$
f(s) \leq C\sum_{i=1}^m\left(\frac{A_i}{(t-s)^{\alpha_i}}+B_i\right),
$$
where $C$ is a constant depending on $\alpha_1,\alpha_2,\cdots,\alpha_m$ and $\theta$.
\end{Lem}

In order to establish the Liouville type results, we need to estimate several integrals involving $u$ and $B$.
Denote
$$
\aligned
&J_1=\frac{1}{(t-s)^2}\int_{B_t \backslash B_s}(|u|^2+|B|^2) d x,\;J_2=\frac{1}{t-s} \int_{B_t \backslash B_{\frac{3R}{2}}}|u|^3dx,\\
&J_3=\frac{1}{t-s}\|B\|_{L^{2p'}(B_t\backslash B_\frac{3R}{2})}^{2}\|u\|_{L^p(B_t\backslash B_s)},\;J_4=\frac{1}{(t-s)^2}\int_{B_t\backslash B_s}|B|^4dx.
\endaligned
$$
The estimates on $J_1$, $J_2$, $J_3$, $J_4$ are given in the next four lemmas.

\begin{Lem}
Let $\sqrt{3}R\leq s<t\leq 2R$ and $p,q\geq1$. Suppose that $u,B$ are smooth vector-valued functions. Then
\begin{itemize}
\item[(i)] For any $\varepsilon,\delta>0$, there exist positive constants $C_\varepsilon$ and $C_\delta$ such that
\begin{equation}\label{ine2.1}
\aligned
J_1\leq&\varepsilon\|u\|_{L^{6}(B_t\backslash B_s)}^{2}+\frac{C_\varepsilon}{(t-s)^{\frac{6}{p}-1}}\|u\|_{L^{p}(A_R)}^{2}+\frac{C}{(t-s)^2}R^{3-\frac{6}{p}}\|u\|_{L^p (A_R)}^2+\\
&\delta\|B\|_{L^{6}(B_t\backslash B_s)}^{2}+\frac{C_\delta}{(t-s)^{\frac{6}{q}-1}}\|B\|_{L^{q}(A_R)}^{2}+\frac{C}{(t-s)^2}R^{3-\frac{6}{q}}\|B\|_{L^{q}(A_R)}^{2}.
\endaligned
\end{equation}
\item[(ii)] It holds that
\begin{equation}\label{ine2.2}
\aligned
J_1\leq\frac{CR^2}{(t-s)^2}\left(\|u\|_{L^6(A_R)}^2 +\|B\|_{L^6(A_R)}^2\right).
\endaligned
\end{equation}
\end{itemize}
\end{Lem}
\begin{proof}
Denote
\begin{equation*}
J_{11}=\frac{1}{(t-s)^2}\int_{B_t\backslash B_s}|u|^2 dx,\;J_{12}=\frac{1}{(t-s)^2}\int_{B_t\backslash B_s}|B|^2dx,
\end{equation*}
then $J_1=J_{11}+J_{12}$.
If $1\leq p<2$, using the interpolation inequality
$$
\|u\|_{L^2}\leq\|u\|_{L^p}^{\frac{2p}{6-p}}\|u\|_{L^6}^{\frac{6-3p}{6-p}},
$$
and the Young inequality, we have
\begin{equation}\label{ine2.3}
\aligned
J_{11}&\leq\frac{1}{(t-s)^{2}}\|u\|_{L^p(B_t\backslash B_s)}^{\frac{4p}{6-p}}\|u\|_{L^6(B_t\backslash B_s)}^{\frac{12-6p}{6-p}}\\
&\leq\varepsilon\|u\|_{L^{6}(B_t\backslash B_s)}^{2}+\frac{C_\varepsilon}{(t-s)^{\frac{6}{p}-1}}\|u\|_{L^{p}(B_t\backslash B_s)}^{2}\\
&\leq\varepsilon\|u\|_{L^{6}(B_t\backslash B_s)}^{2}+\frac{C_\varepsilon}{(t-s)^{\frac{6}{p}-1}}\|u\|_{L^{p}(A_R)}^{2}.
\endaligned
\end{equation}
If $p\geq2$, using the H\"{o}lder inequality, we obtain
\begin{equation}\label{ine2.4}
J_{11}\leq\frac{C}{(t-s)^2}R^{3-\frac{6}{p}}\|u\|_{L^p (A_R)}^2.
\end{equation}
Combining \eqref{ine2.3} and \eqref{ine2.4}, we conclude that no matter whether $p$ is larger than $2$ or not,  $J_{11}$ can be always controlled by
\begin{equation}\label{ine2.5}
J_{11}\leq\varepsilon\|u\|_{L^{6}(B_t\backslash B_s)}^{2}+\frac{C_\varepsilon}{(t-s)^{\frac{6}{p}-1}}\|u\|_{L^{p}(A_R)}^{2}+\frac{C}{(t-s)^2}R^{3-\frac{6}{p}}\|u\|_{L^p (A_R)}^2.
\end{equation}
Similarly to the estimate of $J_{11}$, we can get
\begin{equation*}
\aligned
J_{12}\leq\delta\|B\|_{L^{6}(B_t\backslash B_s)}^{2}+\frac{C_\delta}{(t-s)^{\frac{6}{q}-1}}\|B\|_{L^{q}(A_R)}^{2}+\frac{C}{(t-s)^2}R^{3-\frac{6}{q}}\|B\|_{L^{q}(A_R)}^{2}.
\endaligned
\end{equation*}
Consequently, \eqref{ine2.1} holds.

Using the H\"{o}lder inequality, we obtain
\begin{align*}
J_1&=\frac{1}{(t-s)^2}\|u\|_{L^2(B_t\backslash B_s)}^2+\frac{1}{(t-s)^2}\|B\|_{L^2(B_t\backslash B_s)}^2\\
&\leq\frac{Ct^2}{(t-s)^2}\left(\|u\|_{L^6(B_t\backslash B_s)}^2 +\|B\|_{L^6(B_t\backslash B_s)}^2\right)\\
&\leq\frac{CR^2}{(t-s)^2}\left(\|u\|_{L^6(A_R)}^2 +\|B\|_{L^6(A_R)}^2\right).
\end{align*}

\end{proof}

\begin{Lem}
Let $\sqrt{3}R\leq s<t\leq 2R$. Suppose that $u$ is a smooth vector-valued function. Then we have the following conclusions:
\begin{itemize}
\item[(i)]  Assume $p\in\left[1,3\right)$. It holds that
\begin{equation}\label{ine2.6}
J_2\leq \frac{1}{t-s}\|u\|_{L^p(B_t\backslash B_{\frac{3R}{2}})}^{\frac{3p}{6-p}}\|u\|_{L^6(B_t\backslash B_{\frac{3R}{2}})}^{\frac{18-6p}{6-p}}.
\end{equation}
\item[(ii)] Assume $p\in\left(\frac{3}{2},3\right)$. For any $\varepsilon>0$, it holds that
\begin{equation}\label{ine2.7}
\aligned
J_2&\leq\varepsilon\|u\|_{L^6(B_t\backslash B_{\frac{3R}{2}})}^2+\frac{C_\varepsilon}{(t-s)^\frac{6-p}{2p-3}}\|u\|_{L^p(A_R)}^\frac{3p}{2p-3}.
\endaligned
\end{equation}
\item[(iii)]  Assume $p\geq3$. It holds that
\begin{equation}\label{ine2.8}
J_2\leq\frac{C}{t-s}R^{3-\frac{9}{p}}\|u\|_{L^{p}(A_R)}^{3}.
\end{equation}
\end{itemize}
\end{Lem}

\begin{proof}
If $1\leq p<3$, using the interpolation inequality
$$
\|u\|_{L^3 } \leq\|u\|_{L^p }^{\frac{p}{6-p}}\|u\|_{L^6}^{\frac{6-2 p}{6-p}},
$$
we obtain
\begin{equation*}
\aligned
J_2&\leq \frac{1}{t-s}\|u\|_{L^p(B_t\backslash B_{\frac{3R}{2}})}^{\frac{3p}{6-p}}\|u\|_{L^6(B_t\backslash B_{\frac{3R}{2}})}^{\frac{18-6p}{6-p}}.
\endaligned
\end{equation*}
Then by the Young inequality, we find
\begin{equation*}
\aligned
J_2&\leq\varepsilon\|u\|_{L^6(B_t\backslash B_{\frac{3R}{2}})}^2+\frac{C_\varepsilon}{(t-s)^\frac{6-p}{2p-3}}\|u\|_{L^p(A_R)}^\frac{3p}{2p-3},
\endaligned
\end{equation*}
here we require $p>\frac{3}{2}$.

If $p\geq3$, by the H\"{o}lder inequality, we deduce
\begin{equation*}
J_2\leq\frac{C}{t-s}t^{3-\frac{9}{p}}\|u\|_{L^{p}(B_t\backslash B_{\frac{3R}{2}})}^{3}\leq\frac{C}{t-s}R^{3-\frac{9}{p}}\|u\|_{L^{p}(A_R)}^{3}.
\end{equation*}
\end{proof}

\begin{Lem}
Let $\sqrt{3}R\leq s<t\leq 2R$. Suppose that $u,B$ are smooth vector-valued functions. Then
\begin{itemize}
\item[(i)] Let $\frac{3}{2}<p\leq\frac{9}{2}$, $1\leq q<2p'$. For any $\delta>0$, there exists a positive constant $C_\delta$ such that
\begin{equation}\label{ine2.9}
J_3\leq\delta\|B\|_{L^6(B_t\backslash B_{\frac{3R}{2}})}^2+\frac{C_\delta}{(t-s)^\frac{(6-q)p'}{(3-p')q}}\|u\|_{L^p(A_R)}^\frac{(6-q)p'}{(3-p')q}\|B\|_{L^q(A_R)}^2.
\end{equation}
\item[(ii)] Let $\frac{3}{2}<p\leq\frac{9}{2}$, $1\leq q<2p'$. It holds that
\begin{equation}\label{ine2.10}
\aligned
J_3&\leq\frac{1}{t-s}\|u\|_{L^p(B_t\backslash B_s)}\|B\|_{L^q(B_t\backslash B_{\frac{3R}{2}})}^\frac{2(3-p')q}{(6-q)p'}\|B\|_{L^6(B_t\backslash B_{\frac{3R}{2}})}^\frac{12p'-6q}{(6-q)p'}.
\endaligned
\end{equation}
\item[(iii)] Let $\frac{3}{2}<p\leq\frac{9}{2}$, $q\geq 2p'$. It holds that
\begin{equation}\label{ine2.11}
\aligned
J_3&\leq\frac{C}{t-s}R^{3-\frac{3}{p}-\frac{6}{q}}\|B\|_{L^q(A_R)}^{2}\|u\|_{L^p(A_R)}.
\endaligned
\end{equation}
\end{itemize}

\end{Lem}
\begin{proof}

If $\frac{3}{2}<p\leq\frac{9}{2}$, $1\leq q<2p'$, we have $2p'\in(q,6)$. Hence, the following interpolation inequality holds
\begin{equation*}
\|B\|_{L^{2p'}} \leq\|B\|_{L^q}^\frac{(3-p')q}{(6-q)p'}\|B\|_{L^6}^\frac{6p'-3q}{(6-q)p'},
\end{equation*}
which indicates
\begin{equation*}
\aligned
J_3&\leq\frac{1}{t-s}\|u\|_{L^p(B_t\backslash B_s)}\|B\|_{L^q(B_t\backslash B_{\frac{3R}{2}})}^\frac{2(3-p')q}{(6-q)p'}\|B\|_{L^6(B_t\backslash B_{\frac{3R}{2}})}^\frac{12p'-6q}{(6-q)p'}.
\endaligned
\end{equation*}
Using the Young inequality, we get
$$
J_3\leq\delta\|B\|_{L^6(B_t\backslash B_{\frac{3R}{2}})}^2+\frac{C_\delta}{(t-s)^\frac{(6-q)p'}{(3-p')q}}\|u\|_{L^p(B_t\backslash B_s)}^\frac{(6-q)p'}{(3-p')q}\|B\|_{L^q(B_t\backslash B_{\frac{3R}{2}})}^2.
$$

If $\frac{3}{2}<p\leq\frac{9}{2}$, $q\geq 2p'$,
applying the H\"{o}lder inequality to $J_3$, we obtain
\begin{equation*}
\aligned
J_3&\leq\frac{C}{t-s}R^{\frac{3}{p'}-\frac{6}{q}}\|B\|_{L^q(B_t\backslash B_{\frac{3R}{2}})}^{2}\|u\|_{L^p(B_t\backslash B_s)}\\
&=\frac{C}{t-s}R^{3-\frac{3}{p}-\frac{6}{q}}\|B\|_{L^q(B_t\backslash B_{\frac{3R}{2}})}^{2}\|u\|_{L^p(B_t\backslash B_s)}.
\endaligned
\end{equation*}

\end{proof}

\begin{Lem}\label{Lem2.7}
Let $\sqrt{3}R\leq s<t\leq 2R$ and $3<q\leq6$. Suppose that $B$ is a smooth vector-valued function. Define $h(q)=1$ for $3<q<4$ and $h(q)=0$ for $4\leq q\leq 6$. Then
\begin{itemize}
\item[(i)] For any $\delta>0$, there exist positive constants $C_{\delta}$ and $C$ such that
\begin{equation}\label{ine2.12}
\aligned
J_4\leq&\delta\|B\|_{L^6(B_t\backslash B_s)}^2+\frac{C_\delta h(q)}{(t-s)^\frac{6-q}{q-3}}\|B\|_{L^q(A_R)}^\frac{q}{q-3}+\frac{C}{(t-s)^2}R^{3-\frac{12}{q}}\|B\|_{L^q(A_R)}^4.
\endaligned
\end{equation}
\item[(ii)] It holds that
\begin{equation}\label{ine2.13}
J_4\leq \frac{CR}{(t-s)^2}\|B\|_{L^6(A_R)}^4.
\end{equation}
\end{itemize}
\end{Lem}

\begin{proof}
For $3<q<4$, substituting the interpolation inequality
$$
\|B\|_{L^4}\leq\|B\|_{L^q}^\frac{q}{12-2q} \|B\|_{L^6}^\frac{12-3q}{12-2q}
$$
into $J_4$ and using the Young inequality, we have
\begin{equation}\label{ine2.14}
\aligned
J_4&\leq\frac{1}{(t-s)^2}\|B\|_{L^q(B_t \backslash B_s)}^\frac{2q}{6-q} \|B\|_{L^6(B_t \backslash B_s)}^\frac{24-6q}{6-q}\\
&\leq\delta\|B\|_{L^6(B_t\backslash B_s)}^2+\frac{C_{\delta}}{(t-s)^\frac{6-q}{q-3}}\|B\|_{L^q(A_R)}^\frac{q}{q-3}.
\endaligned
\end{equation}
Here we mention that we require $q>3$ in the last step of the above inequality.

For $4\leq q\leq6$, using the H\"{o}lder inequality, we obtain
\begin{equation}\label{ine2.15}
J_4\leq\frac{C}{(t-s)^2}R^{3-\frac{12}{q}}\|B\|_{L^q(B_t \backslash B_s)}^4.
\end{equation}
Taking \eqref{ine2.14} and \eqref{ine2.15} into account together, we conclude that \eqref{ine2.12} holds.
Setting $q=6$ in the above inequality, we find
$$
J_4\leq \frac{CR}{(t-s)^2}\|B\|_{L^6(B_t \backslash B_s)}^4.
$$
\end{proof}

\section{Proof of conventional Liouville type theorems}\label{sec3}

In order to handle the MHD equations and Hall-MHD equations together, we introduce the following equations involving a parameter $\kappa$:
\begin{equation}\label{equ3.1}
  \left\{
    \begin{array}{ll}
     -\Delta u+(u\cdot\nabla) u+\nabla \pi=(B\cdot\nabla)B,  \\
   -\Delta B+\kappa\mathrm{curl} (\mathrm{curl} B \times B)+(u\cdot\nabla) B-(B\cdot\nabla)u=0, \\
    \mathrm{div}u=\mathrm{div}B=0.
\end{array}
  \right.
\end{equation}
When $\kappa=0$, $\eqref{equ3.1}$ reduces to the MHD equations. When $\kappa=1$, $\eqref{equ3.1}$ is the Hall-MHD equations.

Let $\sqrt{3}R\leq s<t\leq 2R$, then we infer $s\geq\frac{\sqrt{3}t}{2}>\frac{3R}{2}$. We introduce a cut-off function $\eta \in C_0^{\infty}\left(\mathbb{R}^{3}\right)$ satisfying
\begin{equation*}
\eta(x)= \begin{cases}1, & |x| <s, \\ 0, & |x| >\frac{s+t}{2},\end{cases}
\end{equation*}
with
$$\text{$0\leq\eta (x)\leq 1$, and $\|\nabla \eta\|_{L^{\infty}} \leq \frac{C}{t-s}$, $\|\nabla^2\eta\|_{L^{\infty}} \leq \frac{C}{(t-s)^2}$.}$$
Since
$$\int_{B_t\backslash B_\frac{\sqrt{3}t}{2}}u\cdot\nabla\eta^{2}dx=\int_{B_t}u\cdot\nabla\eta^{2}dx=\int_{B_t}\nabla\cdot (u\eta^{2})dx=\int_{\partial B_t}\frac{x}{|x|}\cdot (u\eta^{2})dS=0,$$
by Lemma \ref{Lem2.1} and Lemma \ref{Lem2.2}, there exists $w\in W_{0}^{1,r}(B_t\backslash \overline{B_\frac{\sqrt{3}t}{2}})$ such that $w$ satisfies the following equation
\begin{equation*}
\mathrm{div} w=u\cdot\nabla\eta^{2} \text{ in }B_t\backslash \overline{B_\frac{\sqrt{3}t}{2}},
\end{equation*}
with the estimate
\begin{equation}\label{ine3.2}
\aligned
\|\nabla w\|_{L^{r}(B_t\backslash B_\frac{\sqrt{3}t}{2})}\leq C\|u\cdot\nabla\eta^2\|_{L^{r}(B_t\backslash B_\frac{\sqrt{3}t}{2})}\leq\frac{C}{t-s}\|u\|_{L^{r}(B_t\backslash B_s)},
\endaligned
\end{equation}
for any $1<r<+\infty$. We extend $w$ by zero to $B_\frac{\sqrt{3}t}{2}$, then $w\in W_{0}^{1,r}(B_t).$

Multiply both sides of $\eqref{equ3.1}_{1}$ and $\eqref{equ3.1}_{2}$ by $u \eta^2-w$ and $B \eta^2$ respectively, integrate over $B_t$ and apply integration by parts. This procedure yields
\begin{align}\label{ine3.3}
&\int_{B_t}\left(|\nabla u|^{2}\eta^2+|\nabla B|^{2}\eta^2\right) d x\notag\\
=&\frac{1}{2}\int_{B_t}(|u|^2+|B|^2) \Delta \eta^2 d x-\kappa\int_{B_t}\mathrm{curl} B\times B\cdot(\nabla \eta^2\times B)dx\notag\\
&+\frac{1}{2} \int_{B_t}|u|^2 u \cdot \nabla \eta^2 d x+\int_{B_t}\nabla u:\nabla w dx- \int_{B_t}(u \cdot\nabla )w \cdot u dx\\
& + \int_{B_t}(B \cdot\nabla )w\cdot Bdx- \int_{B_t}(u \cdot B)B\cdot\nabla \eta^2 dx+\frac{1}{2} \int_{B_t}|B|^2 u \cdot \nabla \eta^2 d x.\notag
\end{align}
Using the Young inequality, we obtain
\begin{align*}
\left|\kappa\int_{B_t}\mathrm{curl} B\times B\cdot(\nabla \eta^2\times B)dx\right|&\leq C\kappa\int_{B_t}|\nabla B|\eta\cdot|B|^2|\nabla\eta|dx\\
&\leq\frac{1}{2}\int_{B_t}|\nabla B|^2\eta^2dx+C\kappa^2\int_{B_t}|B|^4|\nabla\eta|^2dx,
\end{align*}
which with \eqref{ine3.3} give
\begin{align}\label{ine3.4}
&\int_{B_t}\left(|\nabla u|^{2}\eta^2+|\nabla B|^{2}\eta^2\right) d x\notag\\
\leq&\frac{C}{(t-s)^2}\int_{B_t \backslash B_s}(|u|^2+|B|^2) d x+ \frac{C}{t-s} \int_{B_t \backslash B_s}|u|^3  d x+C\int_{B_t\backslash B_\frac{\sqrt{3}t}{2}}|\nabla u||\nabla w| dx\notag\\
&+C\int_{B_t\backslash B_\frac{\sqrt{3}t}{2}}|u|^{2} |\nabla w|dx+C\int_{B_t\backslash B_\frac{\sqrt{3}t}{2}}|B|^{2}|\nabla w|dx +\frac{C}{t-s}\int_{B_t \backslash B_s}|B|^2|u|dx\\
&+\frac{C\kappa^2}{(t-s)^2}\int_{B_t\backslash B_s}|B|^4dx.\notag
\end{align}
Applying the Gagliardo-Nirenberg inequality
$$\|\varphi\|_{L^6 (B_t)}\leq C\|\nabla \varphi\|_{L^2 (B_t)}\text{ for any } \varphi\in W^{1,2}_0(B_t),$$
we have
\begin{align*}
&\|u \eta\|_{L^6 (B_t)}^2+\|B \eta\|_{L^6 (B_t)}^2\\
\leq&C\left(\|\nabla(u \eta)\|_{L^2 (B_t)}^2+\|\nabla(B \eta)\|_{L^2 (B_t)}^2 \right)\\
\leq&C \left(\|\eta\nabla u\|_{L^2 (B_t)}^2+\|\eta\nabla B\|_{L^2 (B_t)}^2+\|u\otimes\nabla \eta\|_{L^2 (B_t)}^2+\|B\otimes\nabla \eta\|_{L^2 (B_t)}^2 \right).
\end{align*}
Here it is worthy mentioning that the constant in the Gagliardo-Nirenberg inequality is not dependent on the radius $t$, which can be proved by a rescaling argument.
Combining the above inequality and \eqref{ine3.4}, we get
\begin{align}\label{ine3.5}
&\|u \eta\|_{L^6 (B_t)}^2+\|B \eta\|_{L^6 (B_t)}^2+\|\eta\nabla u\|_{L^2 (B_t)}^2+\|\eta\nabla B\|_{L^2 (B_t)}^2\notag\\
\leq&\frac{C}{(t-s)^2}\int_{B_t \backslash B_s}(|u|^2+|B|^2) d x+ \frac{C}{t-s} \int_{B_t \backslash B_s}|u|^3d x+C\int_{B_t\backslash B_\frac{\sqrt{3}t}{2}}|\nabla u|\cdot|\nabla w| dx\notag\\
&+C\int_{B_t\backslash B_\frac{\sqrt{3}t}{2}}|u|^{2} |\nabla w|dx+C\int_{B_t\backslash B_\frac{\sqrt{3}t}{2}}|B|^{2}|\nabla w|dx +\frac{C}{t-s}\int_{B_t \backslash B_s}|B|^2|u|dx\\
&+\frac{C\kappa^2}{(t-s)^2}\int_{B_t\backslash B_s}|B|^4dx.\notag
\end{align}

By the H\"{o}lder inequality and \eqref{ine3.2}, we obtain
\begin{align}\label{ine3.6}
&\frac{C}{t-s} \int_{B_t \backslash B_s}|u|^3d x+C\int_{B_t\backslash B_\frac{\sqrt{3}t}{2}}|u|^{2} |\nabla w|dx\notag\\
\leq&\frac{C}{t-s}\|u\|_{L^{3}(B_t\backslash B_s)}^3+C\|u\|_{L^{3}(B_t\backslash B_\frac{\sqrt{3}t}{2})}^{2}\|\nabla w\|_{L^{3}(B_t\backslash B_\frac{\sqrt{3}t}{2})}\notag\\
\leq&\frac{C}{t-s}\|u\|_{L^{3}(B_t\backslash B_s)}^3+\frac{C}{t-s}\|u\|_{L^{3}(B_t\backslash B_\frac{\sqrt{3}t}{2})}^{2}\|u\|_{L^{3}(B_t\backslash B_s)}\\
\leq&\frac{C}{t-s}\|u\|_{L^{3}(B_t\backslash B_\frac{3R}{2})}^3.\notag
\end{align}
By the Young inequality and \eqref{ine3.2}, we have
\begin{equation}\label{ine3.7}
\aligned
C\int_{B_t\backslash B_\frac{\sqrt{3}t}{2}}|\nabla u|\cdot|\nabla w| dx&\leq\frac{1}{2}\int_{B_t\backslash B_\frac{\sqrt{3}t}{2}}|\nabla u|^2dx+C\int_{B_t\backslash B_\frac{\sqrt{3}t}{2}}|\nabla w|^2dx\\
&\leq\frac{1}{2}\int_{B_t\backslash B_\frac{\sqrt{3}t}{2}}|\nabla u|^2dx+\frac{C}{(t-s)^2}\int_{B_t \backslash B_s}|u|^2dx.
\endaligned
\end{equation}
Using the H\"{o}lder inequality and \eqref{ine3.2} again, we obtain
\begin{align}\label{ine3.8}
&C\int_{B_t\backslash B_\frac{\sqrt{3}t}{2}}|B|^{2}|\nabla w|dx +\frac{C}{t-s}\int_{B_t \backslash B_s}|B|^2|u|dx\notag\\
\leq&C\|B\|_{L^{2p'}(B_t\backslash B_\frac{\sqrt{3}t}{2})}^{2}\|\nabla w\|_{L^p(B_t\backslash B_\frac{\sqrt{3}t}{2})}+\frac{C}{t-s}\|B\|_{L^{2p'}(B_t\backslash B_s)}^{2}\|u\|_{L^p(B_t\backslash B_s)}\notag\\
\leq&\frac{C}{t-s}\|B\|_{L^{2p'}(B_t\backslash B_\frac{\sqrt{3}t}{2})}^{2}\|u\|_{L^p(B_t\backslash B_s)}+\frac{C}{t-s}\|B\|_{L^{2p'}(B_t\backslash B_s)}^{2}\|u\|_{L^p(B_t\backslash B_s)}\\
\leq&\frac{C}{t-s}\|B\|_{L^{2p'}(B_t\backslash B_\frac{3R}{2})}^{2}\|u\|_{L^p(B_t\backslash B_s)}.\notag
\end{align}
Combining \eqref{ine3.5}, \eqref{ine3.6}, \eqref{ine3.7} and \eqref{ine3.8}, we conclude
\begin{align*}
&\|u \eta\|_{L^6 (B_t)}^2+\|B \eta\|_{L^6 (B_t)}^2+\|\eta\nabla u\|_{L^2 (B_t)}^2+\|\eta\nabla B\|_{L^2 (B_t)}^2\notag\\
\leq&\frac{1}{2}\int_{B_t\backslash B_\frac{\sqrt{3}t}{2}}|\nabla u|^2dx+\frac{C}{(t-s)^2}\int_{B_t \backslash B_s}(|u|^2+|B|^2) d x+ \frac{C}{t-s} \int_{B_t \backslash B_\frac{3R}{2}}|u|^3d x\\
&+\frac{C}{t-s}\|B\|_{L^{2p'}(B_t\backslash B_\frac{3R}{2})}^{2}\|u\|_{L^p(B_t\backslash B_s)}+\frac{C\kappa^2}{(t-s)^2}\int_{B_t\backslash B_s}|B|^4dx\notag\\
\leq&\frac{1}{2}\int_{B_t\backslash B_\frac{\sqrt{3}t}{2}}|\nabla u|^2dx+C(J_1+J_2+J_3+\kappa^2J_4),
\end{align*}
where $J_1,J_2,J_3,J_4$ are defined in Section \ref{sec2}. We denote
\begin{equation*}
f(\rho)=\|u\|_{L^6 (B_\rho)}^2+\|B\|_{L^6 (B_\rho)}^2+\|\nabla u\|_{L^2 (B_\rho)}^2+\|\nabla B\|_{L^2 (B_\rho)}^2,
\end{equation*}
and then we have
\begin{align}\label{ine3.9}
f(s)\leq\frac{1}{2}\int_{B_t\backslash B_\frac{\sqrt{3}t}{2}}|\nabla u|^2dx+C(J_1+J_2+J_3+\kappa^2J_4).
\end{align}
It is noted that we will take $\kappa=0$ in the proof of Theorem \ref{main1} and take $\kappa=1$ in the proof of Theorem \ref{main2}.

\begin{proof}[{\bf Proof of Theorem \ref{main1}}]
We divide the assumptions into two main cases, i.e., $\frac{3}{2} <p<3$ and $3 \leq p\leq\frac{9}{2}$.
Firstly, we consider the case $\frac{3}{2} <p<3$. Since
$$\liminf\limits_{R\rightarrow+\infty}\left[X_{p,\alpha}(R)+Y_{q,\beta}(R)\right]<+\infty,$$
there exists a sequence $R_j\nearrow+\infty$ such that
\begin{equation}\label{ine3.10}
 \lim\limits_{j\rightarrow+\infty}X_{p,\alpha}(R_j)<+\infty, \lim\limits_{j\rightarrow+\infty}Y_{q,\beta}(R_j)<+\infty.
\end{equation}

\textbf{Assume that (A1) holds.}
Combining \eqref{ine3.9}, \eqref{ine2.1}, \eqref{ine2.7} and \eqref{ine2.9}, we derive that
\begin{align*}
&f(s)\leq\frac{1}{2}f(t)+\frac{C}{(t-s)^{\frac{6}{p}-1}}\|u\|_{L^p(A_R)}^{2}+\frac{C}{(t-s)^2}R^{3-\frac{6}{p}}\|u\|_{L^{p}(A_R)}^{2}\\
&+\frac{C}{(t-s)^{\frac{6}{q}-1}}\|B\|_{L^{q}(A_R)}^{2}+\frac{C}{(t-s)^2}R^{3-\frac{6}{q}}\|B\|_{L^q(A_R)}^{2}\\
&+\frac{C}{(t-s)^\frac{6-p}{2p-3}}\|u\|_{L^p(A_R)}^\frac{3p}{2p-3}+\frac{C}{(t-s)^\frac{(6-q)p'}{(3-p')q}}\|u\|_{L^p(A_R)}^\frac{(6-q)p'}{(3-p')q}\|B\|_{L^q(A_R)}^2.
\end{align*}
Applying Lemma \ref{Lem2.3} to the above function inequality, and taking $s=\sqrt{3}R$ and $t=2R$, we conclude that
\begin{align}\label{ine3.11}
f(R)\leq& f\left(\sqrt{3}R\right)\leq CR^{1-\frac{6}{p}}\|u\|_{L^p(A_R)}^2+CR^{1-\frac{6}{q}}\|B\|_{L^{q}(A_R)}^{2}\notag\\
&+CR^\frac{p-6}{2p-3}\|u\|_{L^p(A_R)}^\frac{3p}{2p-3}+CR^{-\frac{(6-q)p'}{(3-p')q}}\|u\|_{L^p(A_R)}^\frac{(6-q)p'}{(3-p')q}\|B\|_{L^q(A_R)}^2\notag\\
=&CR^{1-\frac{6}{p}+2\alpha}[X_{p,\alpha}(R)]^2+CR^{1-\frac{6}{q}+2\beta}[Y_{q,\beta}(R)]^2+CR^{\frac{p-6}{2p-3}+\frac{3p}{2p-3}\alpha}[X_{p,\alpha}(R)]^\frac{3p}{2p-3}\\
&+CR^{-\frac{(6-q)p'}{(3-p')q}+\frac{(6-q)p'}{(3-p')q}\alpha+2\beta}[X_{p,\alpha}(R)]^\frac{(6-q)p'}{(3-p')q}[Y_{q,\beta}(R)]^2.\notag
\end{align}
Letting $R=R_j\rightarrow+\infty$, using \eqref{ine3.10}, and considering
$$1-\frac{6}{p}+2\alpha<0,\;1-\frac{6}{q}+2\beta\leq0,\;\frac{p-6}{2p-3}+\frac{3p}{2p-3}\alpha\leq0,$$
and
$$-\frac{(6-q)p'}{(3-p')q}+\frac{(6-q)p'}{(3-p')q}\alpha+2\beta=\frac{(6-q)p}{(2p-3)q}\left(-1+\alpha+\frac{(4p-6)q}{(6-q)p}\beta\right)\leq0,$$
we get $u,B\in L^6(\mathbb{R}^3)$ and $|\nabla u|,|\nabla B|\in L^2(\mathbb{R}^3)$. Furthermore, we have
\begin{equation}\label{lim3.12}
\aligned
&\lim_{R\rightarrow+\infty}\|u\|_{L^6(A_R)}=\lim_{R\rightarrow+\infty}\|B\|_{L^6(A_R)}=0,\\
&\lim_{R\rightarrow+\infty}\|\nabla u\|_{L^2(A_R)}=\lim_{R\rightarrow+\infty}\|\nabla B\|_{L^2(A_R)}=0.
\endaligned
\end{equation}
Combining \eqref{ine3.9}, \eqref{ine2.2}, \eqref{ine2.6} and \eqref{ine2.10}, we have
\begin{align*}
f(s)\leq&\frac{1}{2}\|\nabla u\|_{L^2(A_R)}^2+\frac{CR^2}{(t-s)^2}\left(\|u\|_{L^6(A_R)}^2 +\|B\|_{L^6(A_R)}^2\right)\\
&+\frac{C}{t-s}\|u\|_{L^p(A_R)}^{\frac{3p}{6-p}}\|u\|_{L^6(A_R)}^{\frac{18-6p}{6-p}}+\frac{C}{t-s}\|u\|_{L^p(A_R)}\|B\|_{L^q(A_R)}^\frac{2(3-p')q}{(6-q)p'}\|B\|_{L^6(A_R)}^\frac{12p'-6q}{(6-q)p'}.
\end{align*}
Taking $s=\sqrt{3}R$ and $t=2R$ in the above inequality, we get
\begin{align*}
f(R)\leq& f\left(\sqrt{3}R\right)\leq\frac{1}{2}\|\nabla u\|_{L^{2}(A_R)}^2+C\left(\|u\|_{L^6(A_R)}^2 +\|B\|_{L^6(A_R)}^2\right)\\
&+CR^{-1}\|u\|_{L^p(A_R)}^{\frac{3p}{6-p}}\|u\|_{L^6(A_R)}^{\frac{18-6p}{6-p}}+CR^{-1}\|u\|_{L^p(A_R)}\|B\|_{L^q(A_R)}^\frac{2(3-p')q}{(6-q)p'}\|B\|_{L^6(A_R)}^\frac{12p'-6q}{(6-q)p'}\\
=&\frac{1}{2}\|\nabla u\|_{L^{2}(A_R)}^2+C\left(\|u\|_{L^6(A_R)}^2 +\|B\|_{L^6(A_R)}^2\right)+CR^{\frac{3p}{6-p}\alpha-1}[X_{p,\alpha}(R)]^{\frac{3p}{6-p}}\|u\|_{L^6(A_R)}^{\frac{18-6p}{6-p}}\\
&+CR^{-1+\alpha+\frac{2(3-p')q}{(6-q)p'}\beta}X_{p,\alpha}(R)[Y_{q,\beta}(R)]^\frac{2(3-p')q}{(6-q)p'}\|B\|_{L^6(A_R)}^\frac{12p'-6q}{(6-q)p'}.
\end{align*}
Letting $R=R_j\rightarrow+\infty$ and using \eqref{lim3.12}, we find that $u=B=0$.

\textbf{Assume that (A2) holds.}
Routinely, combining \eqref{ine3.9}, \eqref{ine2.1}, \eqref{ine2.7} and \eqref{ine2.11}, and applying Lemma \ref{Lem2.3}, we derive that
\begin{align}\label{ine3.13}
f(R)\leq&CR^{1-\frac{6}{p}}\|u\|_{L^p(A_R)}^2+CR^{1-\frac{6}{q}}\|B\|_{L^{q}(A_R)}^{2}\notag\\
&+CR^\frac{p-6}{2p-3}\|u\|_{L^p(A_R)}^\frac{3p}{2p-3}+CR^{2-\frac{3}{p}-\frac{6}{q}}\|B\|_{L^q(A_R)}^{2}\|u\|_{L^p(A_R)}\notag\\
=&CR^{1-\frac{6}{p}+2\alpha}[X_{p,\alpha}(R)]^2+CR^{1-\frac{6}{q}+2\beta}[Y_{q,\beta}(R)]^2\\
&+CR^{\frac{p-6}{2p-3}+\frac{3p}{2p-3}\alpha}[X_{p,\alpha}(R)]^\frac{3p}{2p-3}+CR^{2-\frac{3}{p}-\frac{6}{q}+\alpha+2\beta}X_{p,\alpha}(R)[Y_{q,\beta}(R)]^2.\notag
\end{align}
Combining \eqref{ine3.9}, \eqref{ine2.2}, \eqref{ine2.6} and \eqref{ine2.11}, we get
\begin{align}\label{inea3.14}
f(R)\leq&\frac{1}{2}\|\nabla u\|_{L^{2}(A_R)}^2+C\left(\|u\|_{L^6(A_R)}^2 +\|B\|_{L^6(A_R)}^2\right)\notag\\
&+\frac{C}{R}\|u\|_{L^p(A_R)}^{\frac{3p}{6-p}}\|u\|_{L^6(A_R)}^{\frac{18-6p}{6-p}}+CR^{2-\frac{3}{p}-\frac{6}{q}}\|B\|_{L^q(A_R)}^{2}\|u\|_{L^p(A_R)}\\
=&\frac{1}{2}\|\nabla u\|_{L^{2}(A_R)}^2+C\left(\|u\|_{L^6(A_R)}^2 +\|B\|_{L^6(A_R)}^2\right)\notag\\
&+CR^{\frac{3p}{6-p}\alpha-1}[X_{p,\alpha}(R)]^{\frac{3p}{6-p}}\|u\|_{L^6(A_R)}^{\frac{18-6p}{6-p}}+CR^{2-\frac{3}{p}-\frac{6}{q}+\alpha
+2\beta}X_{p,\alpha}(R)[Y_{q,\beta}(R)]^2.\notag
\end{align}
Consequently, $u=B=0$.

Next, we consider the case $3 \leq p\leq\frac{9}{2}$.
Since
 $$
 \liminf\limits_{R\rightarrow+\infty}X_{p,\alpha}(R)=0,\;\limsup\limits_{R\rightarrow+\infty}Y_{q,\beta}(R)<+\infty,
 $$
there exists a sequence $R_j\nearrow+\infty$ such that
\begin{equation}\label{ine3.14}
 \lim\limits_{j\rightarrow+\infty}X_{p,\alpha}(R_j)=0,\;\lim\limits_{j\rightarrow+\infty}Y_{q,\beta}(R_j)<+\infty.
\end{equation}

\textbf{Assume that (A3) holds.} Combining \eqref{ine3.9}, \eqref{ine2.1}, \eqref{ine2.8} and \eqref{ine2.9}, we derive that
\begin{align*}
&f(s)\leq\frac{1}{2}f(t)+\frac{C}{(t-s)^{\frac{6}{p}-1}}\|u\|_{L^p(A_R)}^{2}+\frac{C}{(t-s)^2}R^{3-\frac{6}{p}}\|u\|_{L^{p}(A_R)}^{2}\\
&+\frac{C}{(t-s)^{\frac{6}{q}-1}}\|B\|_{L^{q}(A_R)}^{2}+\frac{C}{(t-s)^2}R^{3-\frac{6}{q}}\|B\|_{L^q(A_R)}^{2}\\
&+\frac{C}{t-s}R^{3-\frac{9}{p}}\|u\|_{L^{p}(A_R)}^{3}+\frac{C}{(t-s)^\frac{(6-q)p'}{(3-p')q}}\|u\|_{L^p(A_R)}^\frac{(6-q)p'}{(3-p')q}\|B\|_{L^q(A_R)}^2.
\end{align*}
Applying Lemma \ref{Lem2.3} to the above function inequality, and taking $s=\sqrt{3}R$ and $t=2R$, we conclude that
\begin{align}\label{ine3.15}
f(R)
\leq&f\left(\sqrt{3}R\right)\leq CR^{1-\frac{6}{p}}\|u\|_{L^{p}(A_R)}^{2}+CR^{1-\frac{6}{q}}\|B\|_{L^{q}(A_R)}^{2}\notag\\
&+CR^{2-\frac{9}{p}}\|u\|_{L^{p}(A_R)}^{3}+CR^{-\frac{(6-q)p'}{(3-p')q}}\|u\|_{L^p(A_R)}^\frac{(6-q)p'}{(3-p')q}\|B\|_{L^q(A_R)}^2\notag\\
=&CR^{1-\frac{6}{p}+2\alpha}[X_{p,\alpha}(R)]^2+CR^{1-\frac{6}{q}+2\beta}[Y_{q,\beta}(R)]^2+CR^{2-\frac{9}{p}+3\alpha}[X_{p,\alpha}(R)]^3\\
&+CR^{-\frac{(6-q)p'}{(3-p')q}+\frac{(6-q)p'}{(3-p')q}\alpha+2\beta}[X_{p,\alpha}(R)]^\frac{(6-q)p'}{(3-p')q}[Y_{q,\beta}(R)]^2.\notag
\end{align}
Letting $R=R_j\rightarrow+\infty$, using \eqref{ine3.14}, and considering
$$1-\frac{6}{p}+2\alpha<0,\;1-\frac{6}{q}+2\beta\leq0,\;2-\frac{9}{p}+3\alpha\leq0,$$
and
$$-\frac{(6-q)p'}{(3-p')q}+\frac{(6-q)p'}{(3-p')q}\alpha+2\beta=\frac{(6-q)p}{(2p-3)q}\left(-1+\alpha+\frac{(4p-6)q}{(6-q)p}\beta\right)\leq0,$$
we get $u,B\in L^6(\mathbb{R}^3)$ and $|\nabla u|,|\nabla B|\in L^2(\mathbb{R}^3)$. Furthermore, \eqref{lim3.12} holds.
Combining \eqref{ine3.9}, \eqref{ine2.2}, \eqref{ine2.8} and \eqref{ine2.9}, we obtain
\begin{align*}
f(s)&\leq\frac{1}{2}\|\nabla u\|_{L^2(A_R)}^2+\frac{CR^2}{(t-s)^2}\left(\|u\|_{L^6(A_R)}^2 +\|B\|_{L^6(A_R)}^2\right)+\frac{C}{t-s}R^{3-\frac{9}{p}}\|u\|_{L^{p}(A_R)}^{3}\\
&+\frac{1}{2}\|B\|_{L^6(A_R)}^2+\frac{C}{(t-s)^\frac{(6-q)p'}{(3-p')q}}\|u\|_{L^p(A_R)}^\frac{(6-q)p'}{(3-p')q}\|B\|_{L^q(A_R)}^2.
\end{align*}
Taking $s=\sqrt{3}R$ and $t=2R$ in the above inequality, we get
\begin{align*}
f(R)\leq& f\left(\sqrt{3}R\right)\leq \frac{1}{2}\|\nabla u\|_{L^2(A_R)}^2+C\left(\|u\|_{L^6(A_R)}^2 +\|B\|_{L^6(A_R)}^2\right)\\
&+CR^{2-\frac{9}{p}}\|u\|_{L^{p}(A_R)}^{3}+CR^{-\frac{(6-q)p'}{(3-p')q}}\|u\|_{L^p(A_R)}^\frac{(6-q)p'}{(3-p')q}\|B\|_{L^q(A_R)}^2\\
=&\frac{1}{2}\|\nabla u\|_{L^2(A_R)}^2+C\left(\|u\|_{L^6(A_R)}^2 +\|B\|_{L^6(A_R)}^2\right)+CR^{2-\frac{9}{p}+3\alpha}[X_{p,\alpha}(R)]^3\\
&+CR^{-\frac{(6-q)p'}{(3-p')q}+\frac{(6-q)p'}{(3-p')q}\alpha+2\beta}[X_{p,\alpha}(R)]^\frac{(6-q)p'}{(3-p')q}[Y_{q,\beta}(R)]^2.
\end{align*}
Letting $R=R_j\rightarrow+\infty$ and considering \eqref{lim3.12}, we find that $u=B=0$.

\textbf{Assume that (A4) holds.} Routinely, combining \eqref{ine3.9}, \eqref{ine2.1}, \eqref{ine2.8} and \eqref{ine2.11}, and applying Lemma \ref{Lem2.3}, we deduce that
\begin{align*}
f(R)\leq&CR^{1-\frac{6}{p}}\|u\|_{L^p(A_R)}^2+CR^{1-\frac{6}{q}}\|B\|_{L^{q}(A_R)}^{2}\\
&+CR^{2-\frac{9}{p}}\|u\|_{L^{p}(A_R)}^{3}+CR^{2-\frac{3}{p}-\frac{6}{q}}\|B\|_{L^q(A_R)}^{2}\|u\|_{L^p(A_R)}\\
=&CR^{1-\frac{6}{p}+2\alpha}[X_{p,\alpha}(R)]^2+CR^{1-\frac{6}{q}+2\beta}[Y_{q,\beta}(R)]^2+CR^{2-\frac{9}{p}+3\alpha}[X_{p,\alpha}(R)]^3\\
&+CR^{2-\frac{3}{p}-\frac{6}{q}+\alpha+2\beta}X_{p,\alpha}(R)[Y_{q,\beta}(R)]^2.
\end{align*}
Combining \eqref{ine3.9}, \eqref{ine2.2}, \eqref{ine2.8} and \eqref{ine2.11}, we have
\begin{align*}
f(R)\leq&\frac{1}{2}\|\nabla u\|_{L^2(A_R)}^2+C\left(\|u\|_{L^6(A_R)}^2 +\|B\|_{L^6(A_R)}^2\right)\\
&+CR^{2-\frac{9}{p}}\|u\|_{L^{p}(A_R)}^{3}+CR^{2-\frac{3}{p}-\frac{6}{q}}\|u\|_{L^{p}(A_R)}\|B\|_{L^q (A_R)}^{2}\\
=&\frac{1}{2}\|\nabla u\|_{L^2(A_R)}^2+C\left(\|u\|_{L^6(A_R)}^2 +\|B\|_{L^6(A_R)}^2\right)\\
&+CR^{2-\frac{9}{p}+3\alpha}[X_{p,\alpha}(R)]^3+CR^{2-\frac{3}{p}-\frac{6}{q}+\alpha+2\beta}X_{p,\alpha}(R)[Y_{q,\beta}(R)]^2.
\end{align*}
Consequently, $u=B=0$.

\end{proof}

\begin{proof}[{\bf Proof of Theorem \ref{main2}}]
For the Hall-MHD equations, we take $\kappa=1$, so we need to deal with the additional term $J_4$, which
can be estimated by Lemma \ref{Lem2.7}.

\textbf{Assume that (B1) holds.} Routinely, combining \eqref{ine3.9}, \eqref{ine2.1}, \eqref{ine2.7}, \eqref{ine2.9}  and \eqref{ine2.12}, and applying Lemma \ref{Lem2.3}, we deduce that
\begin{align*}
f(R)\leq&CR^{1-\frac{6}{p}+2\alpha}[X_{p,\alpha}(R)]^2+CR^{1-\frac{6}{q}+2\beta}[Y_{q,\beta}(R)]^2+CR^{\frac{p-6}{2p-3}+\frac{3p}{2p-3}\alpha}[X_{p,\alpha}(R)]^\frac{3p}{2p-3}\\
&+CR^{-\frac{(6-q)p'}{(3-p')q}+\frac{(6-q)p'}{(3-p')q}\alpha+2\beta}[X_{p,\alpha}(R)]^\frac{(6-q)p'}{(3-p')q}[Y_{q,\beta}(R)]^2\\
&+Ch(q)R^{-\frac{6-q}{q-3}+\frac{q}{q-3}\beta}[Y_{q,\beta}(R)]^\frac{q}{q-3}+CR^{1-\frac{12}{q}+4\beta}[Y_{q,\beta}(R)]^4.
\end{align*}
Noticing that
$$\text{if }\beta\in\left[0,\frac{3}{q}-\frac{1}{2}\right],\text{ then }-\frac{6-q}{q-3}+\frac{q}{q-3}\beta\leq-\frac{6-q}{2q-6},\;1-\frac{12}{q}+4\beta\leq-1;$$
we find $u,B\in L^6(\mathbb{R}^3)$ and $|\nabla u|,|\nabla B|\in L^2(\mathbb{R}^3)$. Furthermore, \eqref{lim3.12} holds.
Combining \eqref{ine3.9}, \eqref{ine2.2}, \eqref{ine2.6}, \eqref{ine2.10} and \eqref{ine2.13}, we have
\begin{align*}
f(R)\leq&\frac{1}{2}\|\nabla u\|_{L^{2}(A_R)}^2+C\left(\|u\|_{L^6(A_R)}^2 +\|B\|_{L^6(A_R)}^2\right)+CR^{\frac{3p}{6-p}\alpha-1}[X_{p,\alpha}(R)]^{\frac{3p}{6-p}}\|u\|_{L^6(A_R)}^{\frac{18-6p}{6-p}}\\
&+CR^{-1+\alpha+\frac{2(3-p')q}{(6-q)p'}\beta}X_{p,\alpha}(R)[Y_{q,\beta}(R)]^\frac{2(3-p')q}{(6-q)p'}\|B\|_{L^6(A_R)}^\frac{12p'-6q}{(6-q)p'}+CR^{-1}\|B\|_{L^6(A_R)}^4.
\end{align*}
Consequently, $u=B=0$.

\textbf{Assume that (B2) holds.}
Routinely, combining \eqref{ine3.9}, \eqref{ine2.1}, \eqref{ine2.7}, \eqref{ine2.11} and \eqref{ine2.12}, and applying Lemma \ref{Lem2.3}, we derive that
\begin{align*}
f(R)\leq&CR^{1-\frac{6}{p}+2\alpha}[X_{p,\alpha}(R)]^2+CR^{1-\frac{6}{q}+2\beta}[Y_{q,\beta}(R)]^2+CR^{\frac{p-6}{2p-3}+\frac{3p}{2p-3}\alpha}[X_{p,\alpha}(R)]^\frac{3p}{2p-3}\\
&+CR^{2-\frac{3}{p}-\frac{6}{q}+\alpha+2\beta}X_{p,\alpha}(R)[Y_{q,\beta}(R)]^2+Ch(q)R^{-\frac{6-q}{q-3}+\frac{q}{q-3}\beta}[Y_{q,\beta}(R)]^\frac{q}{q-3}\\
&+CR^{1-\frac{12}{q}+4\beta}[Y_{q,\beta}(R)]^4,
\end{align*}
which implies that $u,B\in L^6(\mathbb{R}^3)$ and $|\nabla u|,|\nabla B|\in L^2(\mathbb{R}^3)$, and \eqref{lim3.12} holds.
Combining \eqref{ine3.9}, \eqref{ine2.2}, \eqref{ine2.6}, \eqref{ine2.11} and \eqref{ine2.13}, we have
\begin{align*}
f(R)\leq&\frac{1}{2}\|\nabla u\|_{L^{2}(A_R)}^2+C\left(\|u\|_{L^6(A_R)}^2 +\|B\|_{L^6(A_R)}^2\right)+CR^{\frac{3p}{6-p}\alpha-1}[X_{p,\alpha}(R)]^{\frac{3p}{6-p}}\|u\|_{L^6(A_R)}^{\frac{18-6p}{6-p}}\\
&+CR^{2-\frac{3}{p}-\frac{6}{q}+\alpha+2\beta}X_{p,\alpha}(R)[Y_{q,\beta}(R)]^2+CR^{-1}\|B\|_{L^6(A_R)}^4.
\end{align*}
Consequently, $u=B=0$.

\textbf{Assume that (B3) holds.}
Routinely, combining \eqref{ine3.9}, \eqref{ine2.1}, \eqref{ine2.8}, \eqref{ine2.11} and \eqref{ine2.12}, and applying Lemma \ref{Lem2.3}, we demonstrate that
\begin{align*}
f(R)
\leq&CR^{1-\frac{6}{p}+2\alpha}[X_{p,\alpha}(R)]^2+CR^{1-\frac{6}{q}+2\beta}[Y_{q,\beta}(R)]^2+CR^{2-\frac{9}{p}+3\alpha}[X_{p,\alpha}(R)]^3\\
&+CR^{2-\frac{3}{p}-\frac{6}{q}+\alpha+2\beta}X_{p,\alpha}(R)[Y_{q,\beta}(R)]^2+Ch(q)R^{-\frac{6-q}{q-3}+\frac{q}{q-3}\beta}[Y_{q,\beta}(R)]^\frac{q}{q-3}\\
&+CR^{1-\frac{12}{q}+4\beta}[Y_{q,\beta}(R)]^4.
\end{align*}
which implies that $u,B\in L^6(\mathbb{R}^3)$ and $|\nabla u|,|\nabla B|\in L^2(\mathbb{R}^3)$, and \eqref{lim3.12} holds.
Combining \eqref{ine3.9}, \eqref{ine2.2}, \eqref{ine2.8}, \eqref{ine2.11} and \eqref{ine2.13}, we have
\begin{align*}
f(R)\leq &\frac{1}{2}\|\nabla u\|_{L^2(A_R)}+C\left(\|u\|_{L^6(A_R)}^2 +\|B\|_{L^6(A_R)}^2\right)+CR^{2-\frac{9}{p}+3\alpha}[X_{p,\alpha}(R)]^3\\
&+CR^{2-\frac{3}{p}-\frac{6}{q}+\alpha+2\beta}X_{p,\alpha}(R)[Y_{q,\beta}(R)]^2+CR^{-1}\|B\|_{L^6(A_R)}^4.
\end{align*}
Consequently, $u=B=0$.
\end{proof}

\begin{proof}[{\bf Proof of Remark \ref{Rem1.3}}]
We only prove the case of the MHD equations. Since the proof is similar to that of (A2), we only show the difference.
Applying the H\"{o}lder inequality to $\|B\|_{L^q(A_R)}$, and then substituting it into \eqref{inea3.14}, we find that
\begin{equation*}
\aligned
f(R)\leq&\frac{1}{2}\|\nabla u\|_{L^{2}(A_R)}^2+C\left(\|u\|_{L^6(A_R)}^2+\|B\|_{L^6(A_R)}^2\right)\\
&+CR^{\frac{3p}{6-p}\alpha-1}[X_{p,\alpha}(R)]^{\frac{3p}{6-p}}\|u\|_{L^6(A_R)}^{\frac{18-6p}{6-p}}+CR^{1-\frac{3}{p}+\alpha}X_{p,\alpha}(R)\|B\|_{L^6(A_R)}^2.
\endaligned
\end{equation*}
As a consequence, $u=B=0$.
\end{proof}

\begin{Rem}\label{Rem3.1}
Both $\mathrm{(A2)}$ and $\mathrm{(B2)}$ can be replaced by the following assumption
$$\mathrm{(A6)}\;p\in\left(\frac{3}{2},3\right),\;q\in[2p',6],\;\alpha\in\left[0,\frac{2}{p}-\frac{1}{3}\right],\;\beta\in\left[0,\frac{3}{q}-\frac{1}{2}\right],\;\alpha+2\beta=\frac{3}{p}+\frac{6}{q}-2,$$
but the price is that we need to assume in addition that
$$\liminf\limits_{R\rightarrow+\infty}X_{p,\alpha}(R)=0,\;\limsup\limits_{R\rightarrow+\infty}Y_{q,\beta}(R)<+\infty,\text{ or }$$
$$\limsup\limits_{R\rightarrow+\infty}X_{p,\alpha}(R)<+\infty,\;\liminf\limits_{R\rightarrow+\infty}Y_{q,\beta}(R)=0.$$
\end{Rem}

\begin{proof}[{\bf Proof of Remark \ref{Rem3.1}}]
We only prove the case of the MHD equations. In view of the decisive term $CR^{2-\frac{3}{p}-\frac{6}{q}+\alpha+2\beta}X_{p,\alpha}(R)[Y_{q,\beta}(R)]^2$
in \eqref{inea3.14}, we see that the conclusion holds.
\end{proof}

\begin{proof}[{\bf Proof of Corollary \ref{Cor1.7}}]
Since $u\in L^p(\mathbb{R}^3)$, $B\in L^q(\mathbb{R}^3)$, we obtain
$$\lim_{R\rightarrow+\infty}\|u\|_{L^p\left(A_R\right)}=\lim_{R\rightarrow+\infty}\|B\|_{L^q\left(A_R\right)}=0.$$
We divide the range of $p,q$ into four parts, i.e.,
$$
\aligned
&\text{Case I: $\frac{3}{2} <p<3$, $1\leq q<2p'$,}\;&\text{Case II: $\frac{3}{2} <p<3$, $2p'\leq q\leq6$,}\\
&\text{Case III: $3 \leq p\leq\frac{9}{2}$, $1\leq q<2p'$,}\;&\text{Case IV: $3 \leq p\leq\frac{9}{2}$, $2p'\leq q\leq6$.}
\endaligned
$$
By Theorem \ref{main1}, we find that Liouville type theorem holds for Case I and Case III without any additional conditions. Moreover, for Case I and Case III, we have
$$
\frac{1}{p}+\frac{2}{q}>\frac{1}{p}+\frac{2}{2p'}=1>\frac{2}{3}.
$$
By Theorem \ref{main1} and Remark \ref{Rem3.1}, we see that Liouville type theorem holds for Case II and Case IV with the condition
$$
\frac{1}{p}+\frac{2}{q}\geq\frac{2}{3}.
$$

\end{proof}

\begin{proof}[{\bf Proof of Corollary \ref{Cor1.8}}]
The proof is similar to that of Corollary \ref{Cor1.7}.  The difference is that we require $q>3$.
\end{proof}

\vspace {0.1cm}

\section{Proof of Logarithmic improvement of Liouville type theorems}\label{sec4}
In this section, let $\eta$ be a cut-off function defined by
\begin{align*}
\eta(x)= \begin{cases}
1, & |x| <\frac{3R}{2}, \\
4-\frac{2}{R}|x|,& \frac{3R}{2}\leq |x|\leq 2R,\\
0, & |x| >2R.
\end{cases}
\end{align*}
For any $R>0$, we define
\begin{equation*}
E(R)=\int_{\mathbb{R}^3}\left(|\nabla u|^{2}\eta+|\nabla B|^{2}\eta\right) d x.
\end{equation*}
We rewrite $E(R)$ as the following form
$$
E(R)=\int_{B_{\frac{3R}{2}}}\left(|\nabla u|^{2}+|\nabla B|^{2}\right) dx+\int_{A_R}\left(|\nabla u|^{2}+|\nabla B|^{2}\right)\left(-\frac{2}{R}|x|+4\right)dx.
$$
By a direct calculation, we obtain
\begin{align}\label{ine4.1}
E'(R)=&\frac{3}{2}\int_{\partial B_{\frac{3R}{2}}}\left(|\nabla u|^{2}+|\nabla B|^{2}\right) dS+\int_{A_R}\left(|\nabla u|^{2}+|\nabla B|^{2}\right)\frac{2}{R^2}|x|dx\notag\\
&+2\int_{\partial B_{2R}}\left(|\nabla u|^{2}+|\nabla B|^{2}\right)\left(-\frac{2}{R}\cdot2R+4\right)dS\notag\\
&-\frac{3}{2}\int_{\partial B_{\frac{3R}{2}}}\left(|\nabla u|^{2}+|\nabla B|^{2}\right)\left(-\frac{2}{R}\cdot\frac{3R}{2}+4\right)dS\\
=&\int_{A_R}\left(|\nabla u|^{2}+|\nabla B|^{2}\right)\frac{2}{R^2}|x|dx\notag\\
\geq&\frac{3}{R}\int_{A_R}\left(|\nabla u|^{2}+|\nabla B|^{2}\right)dx.\notag
\end{align}

The notation of $\overline{\varphi}_R$ represents the mean value of $\varphi$ on the annulus $A_R$. Since
$$\int_{A_R}(u-\overline{u}_R)\cdot\nabla\eta dx=\int_{B_{2R}}(u-\overline{u}_R)\cdot\nabla\eta dx=\int_{B_{2R}}\nabla\cdot [(u-\overline{u}_R)\eta]dx=0,$$
by Lemma \ref{Lem2.1} and Lemma \ref{Lem2.2}, there exists $w\in W_{0}^{1,r}(A_R)$ such that $w$ satisfies the following equation
\begin{align*}
\mathrm{div} w=(u-\overline{u}_R)\cdot\nabla\eta \text{ in }A_R,
\end{align*}
with the estimate
\begin{align}\label{ine4.2}
\aligned
\|\nabla w\|_{L^{r}(A_R)}\leq C\|(u-\overline{u}_R)\cdot\nabla\eta\|_{L^{r}(A_R)}\leq CR^{-1}\|u-\overline{u}_R\|_{L^{r}(A_R)},
\endaligned
\end{align}
for any $1<r<+\infty$. We extend $w$ by zero to $B_\frac{3R}{2}$, then $w\in W_{0}^{1,r}(B_{2R}).$

Let $(u,B)$ be a solution to \eqref{equ3.1}.
Denote $v=u-\overline{u}_R$, $H=B-\overline{B}_R$, respectively, then we find that $(v,H)$ satisfies
\begin{align}\label{equ4.3}
	\left\{
	\begin{array}{ll}
		-\Delta v+(v\cdot\nabla) v+(\overline{u}_R\cdot\nabla) v+\nabla \pi=(H\cdot\nabla) H+(\overline{B}_R\cdot\nabla) H,  \\
		-\Delta H+\kappa\mathrm{curl}(\mathrm{curl}H\times H)+\kappa\mathrm{curl}(\mathrm{curl} H\times \overline{B}_R)+(v\cdot\nabla)H+ \\
        \quad(\overline{u}_R\cdot\nabla)H-(H\cdot\nabla)v-(\overline{B}_R\cdot\nabla)v=0,\\
		\mathrm{div} v=\mathrm{div} H=0.
	\end{array}
	\right.
\end{align}
Denote the $i$-th component of $v$ and $H$ by $v_i$ and $H_i$, respectively. Multiply both sides of $\eqref{equ4.3}_{1}$ and $\eqref{equ4.3}_{2}$ by $v \eta-w$ and $H \eta$ respectively, integrate over $B_{2R}$ and apply integration by parts. This procedure yields
\begin{align}\label{ine4.4}
E&(R)=\int_{B_{2R}}\left(|\nabla v|^{2}\eta+|\nabla H|^{2}\eta\right) d x\notag\\
=&-\int_{B_{2R}}\sum_{i=1}^3 \nabla v_i\cdot (v_i\nabla\eta)d x-\int_{B_{2R}}\sum_{i=1}^3 \nabla H_i\cdot (H_i\nabla\eta)d x+\int_{B_{2R}}\nabla v:\nabla w dx\notag\\
&+\frac{1}{2} \int_{B_{2R}}(|v|^2+|H|^2)v \cdot \nabla \eta d x+\frac{1}{2} \int_{B_{2R}}(|v|^2+|H|^2)\overline{u}_R \cdot \nabla \eta d x\notag\\
&-\kappa\int_{B_{2R}}\mathrm{curl}H\times H\cdot(\nabla \eta\times H)dx-\kappa\int_{B_{2R}}\mathrm{curl}H\times \overline{B}_R\cdot(\nabla \eta\times H)dx\\
&- \int_{B_{2R}}(v \cdot\nabla )w \cdot v dx- \int_{B_{2R}}(\overline{u}_R \cdot\nabla )w \cdot v dx+\int_{B_{2R}}(H \cdot\nabla )w\cdot H dx\notag\\
&+\int_{B_{2R}}(\overline{B}_R \cdot\nabla )w\cdot H dx-\int_{B_{2R}}(v \cdot H)H\cdot\nabla \eta dx-\int_{B_{2R}}(v \cdot H)\overline{B}_R\cdot\nabla \eta dx.\notag
\end{align}
With the help of the H\"{o}lder inequaity and the Poincar\'{e} inequality
$$\|\varphi-\overline{\varphi}_R\|_{L^2(A_R)}\leq CR\|\nabla \varphi\|_{L^2(A_R)}\text{ for any } \varphi\in W^{1,2}(A_R),$$
we have
\begin{align}\label{ine4.5}
&\left|-\int_{B_{2R}}\sum_{i=1}^3 \nabla v_i\cdot (v_i\nabla\eta)d x-\int_{B_{2R}}\sum_{i=1}^3 \nabla H_i\cdot (H_i\nabla\eta)d x\right|\notag\\
\leq& CR^{-1}\|\nabla v\|_{L^2 (A_R)}\|v\|_{L^2 (A_R)}+CR^{-1}\|\nabla H\|_{L^2 (A_R)}\|H\|_{L^2 (A_R)}\\
\leq& C\|\nabla u\|_{L^2 (A_R)}^2+C\|\nabla B\|_{L^2 (A_R)}^2.\notag
\end{align}
Using the H\"{o}lder inequality, \eqref{ine4.2} and the Poincar\'{e} inequality, we get
\begin{align}\label{ine4.6}
\aligned
\left|\int_{B_{2R}}\nabla v:\nabla w dx\right|&\leq\int_{A_R}|\nabla v||\nabla w| dx\\
&\leq\|\nabla v\|_{L^2(A_R)}\|\nabla w\|_{L^2(A_R)}\\
&\leq \|\nabla v\|_{L^2(A_R)}\cdot CR^{-1}\|v\|_{L^2(A_R)}\\
&\leq C\|\nabla u\|_{L^2(A_R)}^2.
\endaligned
\end{align}
Using the Minkowski inequality and the H\"{o}lder inequality, we derive
\begin{align}\label{ine4.7}
\aligned
\|v\|_{L^p(A_R)}&\leq \|u\|_{L^p(A_R)}+\|\overline{u}_R\|_{L^p(A_R)}\\
&\leq\|u\|_{L^p(A_R)}+CR^\frac{3}{p}|\overline{u}_R|\\
&\leq C\|u\|_{L^p(A_R)}.
\endaligned
\end{align}
Similarly, we can derive
\begin{align}\label{ine4.8}
\aligned
\|H\|_{L^q(A_R)}&\leq C\|B\|_{L^q(A_R)}.
\endaligned
\end{align}
By the H\"{o}lder inequality and \eqref{ine4.7}, we have
\begin{align}\label{ine4.9}
\aligned
\left|\int_{B_{2R}}(|v|^2+|H|^2)v \cdot \nabla \eta d x\right|&\leq CR^{-1}\left(\|v\|_{L^3(A_R)}^3+\|v\|_{L^p(A_R)}\|H\|_{L^{2p'}(A_R)}^2\right)\\
&\leq CR^{-1}\left(\|v\|_{L^3(A_R)}^3+\|u\|_{L^p(A_R)}\|H\|_{L^{2p'}(A_R)}^2\right).
\endaligned
\end{align}
By the H\"{o}lder inequality and the Young inequality, we obtain
\begin{align}\label{ine4.10}
&\left|\int_{B_{2R}}(|v|^2+|H|^2)\overline{u}_R \cdot \nabla \eta dx\right|\notag\\
\leq& CR^{-1}\|v\|_{L^3(A_R)}^2\|\overline{u}_R\|_{L^3(A_R)}+CR^{-1}\|\overline{u}_R\|_{L^p(A_R)}\|H\|_{L^{2p'}(A_R)}^2\notag\\
\leq& CR^{-\frac{2}{3}}\|v\|_{L^3(A_R)}^2\cdot R^{\frac{2}{3}-\frac{3}{p}}\|u\|_{L^p(A_R)}+CR^{-1}\|u\|_{L^p(A_R)}\|H\|_{L^{2p'}(A_R)}^2\\
\leq&CR^{-1}\|v\|_{L^3(A_R)}^3+CR^{2-\frac{9}{p}}\|u\|_{L^p(A_R)}^3+CR^{-1}\|u\|_{L^p(A_R)}\|H\|_{L^{2p'}(A_R)}^2.\notag
\end{align}
Applying the H\"{o}lder inequality, the interpolation inequality, the Poincar\'{e} inequality and the Sobolev-Poincar\'{e} inequality
$$\|\varphi-\overline{\varphi}_R\|_{L^6(A_R)}\leq C\|\nabla \varphi\|_{L^2(A_R)}\text{ for any } \varphi\in W^{1,2}(A_R),$$
we obtain
\begin{align}\label{ine4.11}
\left|\kappa\int_{B_{2R}}\mathrm{curl}H\times H\cdot(\nabla \eta\times H)dx\right|\leq& C\kappa R^{-1}\int_{A_R}|\nabla H|\cdot |H|^2dx\notag\\
\leq&C\kappa R^{-1}\|\nabla H\|_{L^2(A_R)}\|H\|_{L^4(A_R)}^2\notag\\
\leq&C\kappa R^{-1}\|\nabla H\|_{L^2(A_R)}\|H\|_{L^2(A_R)}^\frac{1}{2}\|H\|_{L^6(A_R)}^\frac{3}{2}\\
\leq&C\kappa R^{-\frac{1}{2}}\|\nabla B\|_{L^2(A_R)}\cdot\|\nabla B\|_{L^2(A_R)}^2.\notag
\end{align}
Applying the H\"{o}lder inequality and the Poincar\'{e} inequality, we get
\begin{align}\label{ine4.12}
\left|\kappa\int_{B_{2R}}\mathrm{curl}H\times \overline{B}_R\cdot(\nabla \eta\times H)dx\right|\leq& C\kappa R^{-1}\left|\overline{B}_R\right|\int_{A_R}|\nabla H|\cdot |H|dx\notag\\
\leq& C\kappa R^{-1}\left|\overline{B}_R\right|\|\nabla H\|_{L^2(A_R)}\|H\|_{L^2(A_R)}\notag\\
\leq& C\kappa \left|\overline{B}_R\right|\|\nabla H\|_{L^2(A_R)}^2\\
\leq& C\kappa R^{-\frac{3}{q}}\|B\|_{L^q(A_R)}\cdot\|\nabla B\|_{L^2(A_R)}^2.\notag
\end{align}
By the H\"{o}lder inequality and \eqref{ine4.2}, we obtain
\begin{align}\label{ine4.13}
\aligned
\left|\int_{B_{2R}}(v \cdot\nabla )w \cdot v dx\right|\leq\|\nabla w\|_{L^3(A_R)}\|v\|_{L^3(A_R)}^2\leq CR^{-1}\|v\|_{L^3(A_R)}^3.
\endaligned
\end{align}
Using the H\"{o}lder inequality, \eqref{ine4.2} and \eqref{ine4.10}, we have
\begin{align}\label{ine4.14}
\left|\int_{B_{2R}}(\overline{u}_R \cdot\nabla )w \cdot v dx\right|\leq&\|\overline{u}_R\|_{L^3(A_R)}\|\nabla w\|_{L^3(A_R)}\|v\|_{L^3(A_R)}\notag\\
\leq& CR^{-1}\|v\|_{L^3(A_R)}^2\|\overline{u}_R\|_{L^3(A_R)}\\
\leq&CR^{-1}\|v\|_{L^3(A_R)}^3+CR^{2-\frac{9}{p}}\|u\|_{L^p(A_R)}^3.\notag
\end{align}
Using the H\"{o}lder inequality, \eqref{ine4.2} and \eqref{ine4.7}, we have
\begin{align}\label{ine4.15}
&\left|\int_{B_{2R}}(H \cdot\nabla )w\cdot H dx\right|+\left|\int_{B_{2R}}(v \cdot H)H\cdot\nabla \eta dx\right|\notag\\
\leq&\|\nabla w\|_{L^p(A_R)}\|H\|_{L^{2p'}(A_R)}^2+CR^{-1}\|v\|_{L^{p}(A_R)}\|H\|_{L^{2p'}(A_R)}^2\notag\\
\leq&CR^{-1}\|v\|_{L^{p}(A_R)}\|H\|_{L^{2p'}(A_R)}^2\\
\leq&CR^{-1}\|u\|_{L^{p}(A_R)}\|H\|_{L^{2p'}(A_R)}^2,\notag
\end{align}
and
\begin{align}\label{ine4.16}
\aligned
&\left|\int_{B_{2R}}(\overline{B}_R \cdot\nabla )w\cdot H dx\right|+\left|\int_{B_{2R}}(v \cdot H)\overline{B}_R\cdot\nabla \eta dx\right|\\
\leq&\|\overline{B}_R\|_{L^{2p'}(A_R)}\|\nabla w\|_{L^p(A_R)}\|H\|_{L^{2p'}(A_R)}+CR^{-1}\|v\|_{L^{p}(A_R)}\|H\|_{L^{2p'}(A_R)}\|\overline{B}_R\|_{L^{2p'}(A_R)}\\
\leq&CR^{-1}\|v\|_{L^{p}(A_R)}\|H\|_{L^{2p'}(A_R)}\|\overline{B}_R\|_{L^{2p'}(A_R)}\\
\leq&CR^{\frac{1}{2}-\frac{3}{2p}-\frac{3}{q}}\|u\|_{L^{p}(A_R)}\|H\|_{L^{2p'}(A_R)}\|B\|_{L^q(A_R)}.
\endaligned
\end{align}

Combining \eqref{ine4.4}, \eqref{ine4.5}, \eqref{ine4.6}, \eqref{ine4.9}, \eqref{ine4.10}, \eqref{ine4.11}, \eqref{ine4.12}, \eqref{ine4.13}, \eqref{ine4.14}, \eqref{ine4.15} and \eqref{ine4.16}, we conclude
\begin{align}\label{ine4.17}
\aligned
E(R)\leq& C\left(\|\nabla u\|_{L^2(A_R)}^2+\|\nabla B\|_{L^2(A_R)}^2\right)+CR^{-1}\left(\|u\|_{L^{p}(A_R)}\|H\|_{L^{2p'}(A_R)}^2+\|v\|_{L^3(A_R)}^3\right)\\
&+CR^{2-\frac{9}{p}}\|u\|_{L^p(A_R)}^3+CR^{\frac{1}{2}-\frac{3}{2p}-\frac{3}{q}}\|u\|_{L^{p}(A_R)}\|H\|_{L^{2p'}(A_R)}\|B\|_{L^q(A_R)}\\
&+C\kappa\left(R^{-\frac{1}{2}}\|\nabla B\|_{L^2(A_R)}+R^{-\frac{3}{q}}\|B\|_{L^q(A_R)}\right)\|\nabla B\|_{L^2(A_R)}^2.
\endaligned
\end{align}
It is noted that we will take $\kappa=0$ in the proof of Theorem \ref{main3} and take $\kappa=1$ in the proof of Theorem \ref{main4}.

For the sake of simplicity, we denote
\begin{align}
&\gamma_1=2-\frac{9}{p}+3\alpha,\;\gamma_2=2-\frac{3}{p}-\frac{6}{q}+\alpha+2\beta,\label{gam1-2}\\
&\gamma_3=\frac{1}{2}-\frac{3}{2p}-\frac{3}{q}+\alpha+\frac{6p+pq-3q}{(6-q)p}\beta.\label{gam3}
\end{align}

\begin{proof}[{\bf Proof of Theorem \ref{main3}}]
\textbf{Assume that (C1) holds with $\beta\in\left[0,\frac{3}{q}-\frac{1}{2}\right)$.}
Using the interpolation inequality, \eqref{ine4.8} and the Sobolev-Poincar\'{e} inequality, we obtain
\begin{align}\label{ine4.18}
\aligned
CR^{-1}\|u\|_{L^{p}(A_R)}\|H\|_{L^{2p'}(A_R)}^2&\leq CR^{-1}\|u\|_{L^{p}(A_R)}\|H\|_{L^q(A_R)}^\frac{(6-2p')q}{(6-q)p'}\|H\|_{L^6(A_R)}^\frac{12p'-6q}{(6-q)p'}\\
&\leq CR^{-1}\|u\|_{L^{p}(A_R)}\|B\|_{L^q(A_R)}^\frac{(4p-6)q}{(6-q)p}\|\nabla B\|_{L^2(A_R)}^\frac{12p-6q(p-1)}{(6-q)p}.
\endaligned
\end{align}
Using the interpolation inequality, \eqref{ine4.7} and the Sobolev-Poincar\'{e} inequality, we obtain
\begin{align}\label{ine4.19}
\aligned
CR^{-1}\|v\|_{L^3(A_R)}^3&\leq CR^{-1}\|v\|_{L^p(A_R)}^{\frac{3p}{6-p}}\|v\|_{L^6(A_R)}^{\frac{18-6p}{6-p}}\\
&\leq CR^{-1}\|u\|_{L^p(A_R)}^{\frac{3p}{6-p}}\|\nabla u\|_{L^2(A_R)}^{\frac{18-6p}{6-p}}.
\endaligned
\end{align}
By the Young inequality, we have
\begin{align}\label{ine4.20}
&CR^{\frac{1}{2}-\frac{3}{2p}-\frac{3}{q}}\|u\|_{L^{p}(A_R)}\|H\|_{L^{2p'}(A_R)}\|B\|_{L^q(A_R)}\notag\\
=&CR^{-\frac{1}{2}}\|u\|_{L^{p}(A_R)}^\frac{1}{2}\|H\|_{L^{2p'}(A_R)}\cdot R^{1-\frac{3}{2p}-\frac{3}{q}}\|u\|_{L^{p}(A_R)}^\frac{1}{2}\|B\|_{L^q(A_R)}\\
\leq&CR^{-1}\|u\|_{L^{p}(A_R)}\|H\|_{L^{2p'}(A_R)}^2+CR^{2-\frac{3}{p}-\frac{6}{q}}\|u\|_{L^p(A_R)}\|B\|_{L^q(A_R)}^2,\notag
\end{align}
Combining \eqref{ine4.17}, \eqref{ine4.18}, \eqref{ine4.19} and \eqref{ine4.20}, we conclude
\begin{align}\label{ine4.21}
E(R)\leq &C\left(\|\nabla u\|_{L^2(A_R)}^2+\|\nabla B\|_{L^2(A_R)}^2\right)+CR^{-1}\|u\|_{L^p(A_R)}^{\frac{3p}{6-p}}\|\nabla u\|_{L^2(A_R)}^{\frac{18-6p}{6-p}}\notag\\
&+CR^{-1}\|u\|_{L^{p}(A_R)}\|B\|_{L^q(A_R)}^\frac{(4p-6)q}{(6-q)p}\|\nabla B\|_{L^2(A_R)}^\frac{12p-6q(p-1)}{(6-q)p}+CR^{2-\frac{9}{p}}\|u\|_{L^p(A_R)}^3\\
&+CR^{2-\frac{3}{p}-\frac{6}{q}}\|u\|_{L^p(A_R)}\|B\|_{L^q(A_R)}^2.\notag
\end{align}

Since (C1) holds, there exists a constant $R_1>3$ such that the following two inequalities hold for any $R>R_1$:
\begin{equation}\label{ine4.22}
\|u\|_{L^p\left(A_R\right)}\leq CR^\alpha(\ln R)^\lambda,
\end{equation}
\begin{equation}\label{ine4.23}
\|B\|_{L^q\left(A_R\right)}\leq CR^\beta(\ln R)^\mu,
\end{equation}
which with \eqref{ine3.11} imply
\begin{align}\label{ine4.24}
f(R)\leq& CR^{1-\frac{6}{p}+2\alpha}(\ln R)^{2\lambda}+CR^{1-\frac{6}{q}+2\beta}(\ln R)^{2\mu}+CR^{\frac{p-6}{2p-3}+\frac{3p}{2p-3}\alpha}(\ln R)^{\frac{3p}{2p-3}\lambda}\notag\\
&+CR^{-\frac{(6-q)p'}{(3-p')q}+\frac{(6-q)p'}{(3-p')q}\alpha+2\beta}(\ln R)^{\frac{(6-q)p'}{(3-p')q}\lambda}(\ln R)^{2\mu}\notag\\
\leq &C(\ln R)^{\frac{3p}{2p-3}\lambda}+C(\ln R)^{\frac{(6-q)p}{(2p-3)q}\lambda+2\mu}\\
\leq&C(\ln R)^\frac{9-3p}{2p-3}+C(\ln R)^\frac{6p-3q(p-1)}{(2p-3)q},\notag
\end{align}
for any $R\geq R_2$, where $R_2(>R_1)$ is a sufficiently large constant. Here we have used (C1) with $\beta\in\left[0,\frac{3}{q}-\frac{1}{2}\right)$, \eqref{ass1.8} and \eqref{ass1.9}.
Without loss of generality, we assume $p\leq q$. At this moment, we have the following facts in hand:
$$\frac{6p-3q(p-1)}{(2p-3)q}=-\frac{3(p-1)}{2p-3}+\frac{6p}{(2p-3)q}\leq -\frac{3(p-1)}{2p-3}+\frac{6p}{(2p-3)p}=\frac{9-3p}{2p-3},$$
$$\frac{6p-3q(p-1)}{(6-q)p}=\frac{3(p-1)}{p}-\frac{12p-18}{(6-q)p}\leq\frac{3(p-1)}{p}-\frac{12p-18}{(6-p)p}=\frac{9-3p}{6-p}.$$
Hence, it follows from \eqref{ine4.24} that
\begin{equation*}
f(R)\leq C(\ln R)^\frac{9-3p}{2p-3}.
\end{equation*}

Using \eqref{ine4.1} and the above estimate, we have
\begin{align}\label{ine4.25}
&\|\nabla u\|_{L^2(A_R)}^2+\|\nabla B\|_{L^2(A_R)}^2\notag\\
=&\left(\|\nabla u\|_{L^2(A_R)}^2+\|\nabla B\|_{L^2(A_R)}^2\right)^{\frac{2p-3}{6-p}}\left(\|\nabla u\|_{L^2(A_R)}^2+\|\nabla B\|_{L^2(A_R)}^2\right)^{\frac{9-3p}{6-p}}\notag\\
\leq &C[f(2R)]^\frac{2p-3}{6-p}[RE'(R)]^\frac{9-3p}{6-p}\\
\leq& C[R\ln RE'(R)]^\frac{9-3p}{6-p}.\notag
\end{align}
Here we point out that if $p>q$, then we will take advantage of the following estimate
\begin{align*}
&\|\nabla u\|_{L^2(A_R)}^2+\|\nabla B\|_{L^2(A_R)}^2\notag\\
=&\left(\|\nabla u\|_{L^2(A_R)}^2+\|\nabla B\|_{L^2(A_R)}^2\right)^{1-\frac{6p-3q(p-1)}{(6-q)p}}\left(\|\nabla u\|_{L^2(A_R)}^2+\|\nabla B\|_{L^2(A_R)}^2\right)^\frac{6p-3q(p-1)}{(6-q)p}\notag\\
\leq &C[f(2R)]^{1-\frac{6p-3q(p-1)}{(6-q)p}}[RE'(R)]^\frac{6p-3q(p-1)}{(6-q)p}\\
\leq& C[R\ln RE'(R)]^\frac{6p-3q(p-1)}{(6-q)p}.\notag
\end{align*}

Combining \eqref{ine4.21}, \eqref{ine4.22}, \eqref{ine4.23}, \eqref{ine4.25} and \eqref{ine4.1}, we find that
\begin{align}\label{ine4.26}
E(R)\leq& C[R\ln R E'(R)]^\frac{9-3p}{6-p}+CR^{\frac{3p}{6-p}\alpha-1}(\ln R)^{\frac{3p}{6-p}\lambda}[RE'(R)]^\frac{9-3p}{6-p}\notag\\
&+CR^{\alpha+\frac{(4p-6)q}{(6-q)p}\beta-1}(\ln R)^{\lambda+\frac{(4p-6)q}{(6-q)p}\mu}[RE'(R)]^\frac{6p-3q(p-1)}{(6-q)p}\notag\\
&+CR^{\gamma_1}(\ln R)^{3\lambda}+CR^{\gamma_2}(\ln R)^{\lambda+2\mu}\\
\leq&C[R\ln R E'(R)]^\frac{9-3p}{6-p}+C[R\ln R E'(R)]^\frac{6p-3q(p-1)}{(6-q)p}\notag\\
&+CR^{\gamma_1}(\ln R)^{3\lambda}+CR^{\gamma_2}(\ln R)^{\lambda+2\mu},\notag
\end{align}
where $\gamma_1$ and $\gamma_2$ are defined in \eqref{gam1-2}.
Noticing that
$$
\gamma_1\leq2-\frac{9}{p}+3\left(\frac{2}{p}-\frac{1}{3}\right)=1-\frac{3}{p}<0,
$$
and
$$
\aligned
\gamma_2&\leq2-\frac{3}{p}-\frac{6}{q}+1-\frac{(4p-6)q}{(6-q)p}\beta+2\beta\\
&=3-\frac{3}{p}-\frac{6}{q}+\frac{6p-3pq+3q}{(6-q)p}\cdot2\beta\\
&<3-\frac{3}{p}-\frac{6}{q}+\frac{3(p-1)(2p'-q)}{(6-q)p}\cdot2\left(\frac{3}{q}-\frac{1}{2}\right)\\
&=0,
\endaligned
$$
the last two terms in \eqref{ine4.26} are good terms.

We claim that $E(R)\equiv0$, otherwise, in view of the nondecreasing property of $E(R)$, there exists a constant $R_0>R_2$ such that
$$\text{$E(R)\geq E(R_0)>0$ for any $R\geq R_0$.}$$
By the Young inequality, we obtain
\begin{align}\label{ine4.27}
\aligned
C[R\ln R E'(R)]^\frac{6p-3q(p-1)}{(6-q)p}&\leq \frac{1}{4}E(R_0)+C[R\ln R E'(R)]^\frac{9-3p}{6-p}\\
&\leq\frac{1}{4}E(R)+C[R\ln R E'(R)]^\frac{9-3p}{6-p},
\endaligned
\end{align}
which with \eqref{ine4.26} guarantee that
\begin{align}\label{ine4.28}
\aligned
E(R)\leq C[R\ln R E'(R)]^\frac{9-3p}{6-p}+CR^{\gamma_1}(\ln R)^{3\lambda}+CR^{\gamma_2}(\ln R)^{\lambda+2\mu}.
\endaligned
\end{align}
Applying the property of the convex function $t^r$, $r>1$, we derive that
\begin{align}\label{ine4.29}
\aligned
(a+b+c)^r\leq 3^{r-1}(a^r+b^r+c^r),\text{ where }a,b,c\geq0.
\endaligned
\end{align}
Using \eqref{ine4.28} and \eqref{ine4.29}, we get
\begin{align*}
\aligned
E(R)^\frac{6-p}{9-3p}\leq CR\ln R E'(R)+CR^{\frac{6-p}{9-3p}\gamma_1}(\ln R)^{\frac{6-p}{3-p}\lambda}+CR^{\frac{6-p}{9-3p}\gamma_2}(\ln R)^{\frac{6-p}{9-3p}(\lambda+2\mu)}.
\endaligned
\end{align*}
Consequently, it follows that
\begin{align*}
&\ln\ln R-\ln\ln R_0=\int_{R_0}^R\frac{1}{\rho\ln\rho}d\rho\\
\leq &C\int_{R_0}^R\frac{E'(\rho)+\rho^{\frac{6-p}{9-3p}\gamma_1-1}(\ln \rho)^{\frac{6-p}{3-p}\lambda-1}+\rho^{\frac{6-p}{9-3p}\gamma_2-1}(\ln \rho)^{\frac{6-p}{9-3p}(\lambda+2\mu)-1}}{E(\rho)^\frac{6-p}{9-3p}}d\rho\\
\leq&C\int_{R_0}^R\frac{E'(\rho)}{E(\rho)^\frac{6-p}{9-3p}}d\rho+C\int_{R_0}^R\frac{\rho^{\frac{6-p}{9-3p}\gamma_1-1}(\ln \rho)^{\frac{6-p}{3-p}\lambda-1}+\rho^{\frac{6-p}{9-3p}\gamma_2-1}(\ln \rho)^{\frac{6-p}{9-3p}(\lambda+2\mu)-1}}{E(R_0)^{\frac{6-p}{9-3p}}}d\rho\\
\leq &CE(R_0)^{-\frac{2p-3}{9-3p}}+C\int_{R_0}^R\frac{\rho^{\frac{6-p}{9-3p}\gamma_1-1}(\ln \rho)^{\frac{6-p}{3-p}\lambda-1}+\rho^{\frac{6-p}{9-3p}\gamma_2-1}(\ln \rho)^{\frac{6-p}{9-3p}(\lambda+2\mu)-1}}{E(R_0)^{\frac{6-p}{9-3p}}}d\rho\\
<&+\infty.
\end{align*}
Letting $R\rightarrow+\infty$, we get a contradiction. Finally, the following simple inequality
$$\|\nabla u\|_{L^2(B_R)}^2+\|\nabla B\|_{L^2(B_R)}^2\leq E(R)$$
suggests that $u,B$ are constant vectors. Thanks to the condition $$\limsup\limits_{R\rightarrow+\infty}\left[X_{p,\alpha,\lambda}(R)+Y_{q,\beta,\mu}(R)\right]<+\infty,$$
 we conclude that $u=B=0$.

\textbf{Assume that (C1) holds with $\beta=\frac{3}{q}-\frac{1}{2}$.} Since $\beta=\frac{3}{q}-\frac{1}{2}$, it is possible that $\gamma_2=0$. Therefore, we need to estimate the term $CR^{\frac{1}{2}-\frac{3}{2p}-\frac{3}{q}}\|u\|_{L^{p}(A_R)}\|H\|_{L^{2p'}(A_R)}\|B\|_{L^q(A_R)}$ by a new method.
Using the interpolation inequality, \eqref{ine4.8}, and the Sobolev-Poincar\'{e} inequality, we have
\begin{align}\label{ine4.30}
&CR^{\frac{1}{2}-\frac{3}{2p}-\frac{3}{q}}\|u\|_{L^{p}(A_R)}\|H\|_{L^{2p'}(A_R)}\|B\|_{L^q(A_R)}\notag\\
\leq&CR^{\frac{1}{2}-\frac{3}{2p}-\frac{3}{q}}\|u\|_{L^{p}(A_R)}\|B\|_{L^q(A_R)}\|H\|_{L^q(A_R)}^\frac{(3-p')q}{(6-q)p'}\|H\|_{L^6(A_R)}^\frac{6p'-3q}{(6-q)p'}\notag\\
\leq&CR^{\frac{1}{2}-\frac{3}{2p}-\frac{3}{q}}\|u\|_{L^{p}(A_R)}\|B\|_{L^q(A_R)}\|B\|_{L^q(A_R)}^\frac{(2p-3)q}{(6-q)p}\|\nabla B\|_{L^2(A_R)}^\frac{6p-3q(p-1)}{(6-q)p}\\
=&CR^{\frac{1}{2}-\frac{3}{2p}-\frac{3}{q}}\|u\|_{L^{p}(A_R)}\|B\|_{L^q(A_R)}^\frac{6p+pq-3q}{(6-q)p}\|\nabla B\|_{L^2(A_R)}^\frac{6p-3q(p-1)}{(6-q)p}.\notag
\end{align}
Combining \eqref{ine4.17}, \eqref{ine4.18}, \eqref{ine4.19} and \eqref{ine4.30}, we conclude
\begin{align}\label{ine4.31}
E(R)\leq &C\left(\|\nabla u\|_{L^2(A_R)}^2+\|\nabla B\|_{L^2(A_R)}^2\right)+CR^{-1}\|u\|_{L^p(A_R)}^{\frac{3p}{6-p}}\|\nabla u\|_{L^2(A_R)}^{\frac{18-6p}{6-p}}\notag\\
&+CR^{-1}\|u\|_{L^{p}(A_R)}\|B\|_{L^q(A_R)}^\frac{(4p-6)q}{(6-q)p}\|\nabla B\|_{L^2(A_R)}^\frac{12p-6q(p-1)}{(6-q)p}+CR^{2-\frac{9}{p}}\|u\|_{L^p(A_R)}^3\\
&+CR^{\frac{1}{2}-\frac{3}{2p}-\frac{3}{q}}\|u\|_{L^{p}(A_R)}\|B\|_{L^q(A_R)}^\frac{6p+pq-3q}{(6-q)p}\|\nabla B\|_{L^2(A_R)}^\frac{6p-3q(p-1)}{(6-q)p}.\notag
\end{align}
Combining \eqref{ine4.31}, \eqref{ine4.22}, \eqref{ine4.23}, \eqref{ine4.25}, \eqref{ine4.1} and the assumption (C1), we find
\begin{align}\label{ine4.32}
E(R)\leq& C[R\ln R E'(R)]^\frac{9-3p}{6-p}+CR^{\frac{3p}{6-p}\alpha-1}(\ln R)^{\frac{3p}{6-p}\lambda}[RE'(R)]^\frac{9-3p}{6-p}\notag\\
&+CR^{\alpha+\frac{(4p-6)q}{(6-q)p}\beta-1}(\ln R)^{\lambda+\frac{(4p-6)q}{(6-q)p}\mu}[RE'(R)]^\frac{6p-3q(p-1)}{(6-q)p}\notag\\
&+CR^{\gamma_1}(\ln R)^{3\lambda}+CR^{\gamma_3}(\ln R)^{\lambda+\frac{6p+pq-3q}{(6-q)p}\mu}[RE'(R)]^\frac{6p-3q(p-1)}{2(6-q)p}\\
\leq&C[R\ln R E'(R)]^\frac{9-3p}{6-p}+C[R\ln R E'(R)]^\frac{6p-3q(p-1)}{(6-q)p}\notag\\
&+CR^{\gamma_1}(\ln R)^{3\lambda}+C[R\ln RE'(R)]^\frac{6p-3q(p-1)}{2(6-q)p},\notag
\end{align}
where $\gamma_3$ is defined in \eqref{gam3}. In \eqref{ine4.32}, we also have used the fact that
$$
\aligned
\gamma_3&\leq\frac{1}{2}-\frac{3}{2p}-\frac{3}{q}+1-\frac{(4p-6)q}{(6-q)p}\beta+\frac{6p+pq-3q}{(6-q)p}\beta\\
&=\frac{3}{2}-\frac{3}{2p}-\frac{3}{q}+\frac{6p-3pq+3q}{(6-q)p}\cdot\beta\\
&=\frac{3}{2}-\frac{3}{2p}-\frac{3}{q}+\frac{3(p-1)(2p'-q)}{(6-q)p}\cdot \left(\frac{3}{q}-\frac{1}{2}\right)\\
&=0.
\endaligned
$$
By the way, we point out that the following two conditions
$$\lambda+\frac{(4p-6)q}{(6-q)p}\mu\leq\frac{6p-3q(p-1)}{(6-q)p},$$
$$\lambda+\frac{6p+pq-3q}{(6-q)p}\mu\leq\frac{6p-3q(p-1)}{2(6-q)p}$$
are required in \eqref{ine4.32}, and the latter condition is stronger than the former one. Actually, it holds that
$$\lambda+\frac{6p+pq-3q}{(6-q)p}\mu=\lambda+\frac{(4p-6)q}{(6-q)p}\mu+\frac{3(p-1)(2p'-q)}{(6-q)p}\mu.$$

We also claim that $E(R)\equiv0$, otherwise, there exists a constant $R_0>R_2$ such that $E(R_0)>0$.
Using the Young inequality twice, we obtain
\begin{align*}
C[R\ln R E'(R)]^\frac{6p-3q(p-1)}{2(6-q)p}&\leq \frac{1}{4}E(R_0)+C[R\ln R E'(R)]^\frac{6p-3q(p-1)}{(6-q)p}\notag\\
&\leq \frac{1}{4}E(R_0)+\frac{1}{4}E(R_0)+C[R\ln R E'(R)]^\frac{9-3p}{6-p}\\
&\leq\frac{1}{2}E(R)+C[R\ln R E'(R)]^\frac{9-3p}{6-p}.\notag
\end{align*}
Combining \eqref{ine4.32}, \eqref{ine4.27} and the above inequality, we have
\begin{align*}
\aligned
E(R)\leq C[R\ln R E'(R)]^\frac{9-3p}{6-p}+CR^{\gamma_1}(\ln R)^{3\lambda},
\endaligned
\end{align*}
which follows that
\begin{align*}
\aligned
E(R)^\frac{6-p}{9-3p}\leq CR\ln R E'(R)+CR^{\frac{6-p}{9-3p}\gamma_1}(\ln R)^{\frac{6-p}{3-p}\lambda}.
\endaligned
\end{align*}
Consequently, it follows that
\begin{align*}
&\ln\ln R-\ln\ln R_0=\int_{R_0}^R\frac{1}{\rho\ln\rho}d\rho\\
\leq &C\int_{R_0}^R\frac{E'(\rho)+\rho^{\frac{6-p}{9-3p}\gamma_1-1}(\ln \rho)^{\frac{6-p}{3-p}\lambda-1}}{E(\rho)^\frac{6-p}{9-3p}}d\rho\\
\leq&C\int_{R_0}^R\frac{E'(\rho)}{E(\rho)^\frac{6-p}{9-3p}}d\rho+C\int_{R_0}^R\frac{\rho^{\frac{6-p}{9-3p}\gamma_1-1}(\ln \rho)^{\frac{6-p}{3-p}\lambda-1}}{E(R_0)^{\frac{6-p}{9-3p}}}d\rho\\
\leq &CE(R_0)^{-\frac{2p-3}{9-3p}}+C\int_{R_0}^R\frac{\rho^{\frac{6-p}{9-3p}\gamma_1-1}(\ln \rho)^{\frac{6-p}{3-p}\lambda-1}}{E(R_0)^{\frac{6-p}{9-3p}}}d\rho\\
<&+\infty.
\end{align*}
Letting $R\rightarrow+\infty$, we get a contradiction. Therefore, $u=B=0$.

\textbf{Assume that (C2) holds.} Based on (C2), we see that \eqref{ine4.22} and \eqref{ine4.23} still hold.
Using the H\"{o}lder inequality and \eqref{ine4.8}, we get
\begin{align}\label{ine4.33}
\aligned
CR^{-1}\|u\|_{L^{p}(A_R)}\|H\|_{L^{2p'}(A_R)}^2&\leq CR^{2-\frac{3}{p}-\frac{6}{q}}\|u\|_{L^p(A_R)}\|H\|_{L^q(A_R)}^2\\
&\leq CR^{2-\frac{3}{p}-\frac{6}{q}}\|u\|_{L^p(A_R)}\|B\|_{L^q(A_R)}^2.
\endaligned
\end{align}
Thanks to \eqref{ine3.13}, we have
\begin{align}\label{ine4.34}
f(R)\leq&CR^{1-\frac{6}{p}}\|u\|_{L^p(A_R)}^2+CR^{1-\frac{6}{q}}\|B\|_{L^{q}(A_R)}^{2}\notag\\
&+CR^\frac{p-6}{2p-3}\|u\|_{L^p(A_R)}^\frac{3p}{2p-3}+CR^{2-\frac{3}{p}-\frac{6}{q}}\|B\|_{L^q(A_R)}^{2}\|u\|_{L^p(A_R)}\notag\\
\leq&CR^{1-\frac{6}{p}+2\alpha}(\ln R)^{2\lambda}+CR^{1-\frac{6}{q}+2\beta}(\ln R)^{2\mu}+CR^{\frac{p-6}{2p-3}+\frac{3p}{2p-3}\alpha}(\ln R)^{\frac{3p}{2p-3}\lambda}\\
&+CR^{2-\frac{3}{p}-\frac{6}{q}+\alpha+2\beta}(\ln R)^{\lambda+2\mu}\notag\\
\leq& C(\ln R)^{2\mu}+C(\ln R)^\frac{9-3p}{2p-3},\notag
\end{align}
for any $R\geq R_2$, where $R_2(>R_1)$ is a sufficiently large constant.

When $2\mu\leq\frac{9-3p}{2p-3}$, it follows from \eqref{ine4.34} that
$$f(R)\leq C(\ln R)^\frac{9-3p}{2p-3},$$
which guarantees that \eqref{ine4.25} is still valid.
Combining \eqref{ine4.17}, \eqref{ine4.19}, \eqref{ine4.20}, \eqref{ine4.25} and \eqref{ine4.33}, we obtain
\begin{align*}
\aligned
E(R)\leq&  C[R\ln R E'(R)]^\frac{9-3p}{6-p}+CR^{\gamma_1}(\ln R)^{3\lambda}+CR^{\gamma_2}(\ln R)^{\lambda+2\mu}.
\endaligned
\end{align*}
Routinely, we obtain $u=B=0$.

When $2\mu>\frac{9-3p}{2p-3}$, it follows from \eqref{ine4.34} that
$$f(R)\leq C(\ln R)^{2\mu}.$$
Using \eqref{ine4.1} and the above estimate, we have
\begin{align}\label{ine4.35}
&\|\nabla u\|_{L^2(A_R)}^2+\|\nabla B\|_{L^2(A_R)}^2\notag\\
=&\left(\|\nabla u\|_{L^2(A_R)}^2+\|\nabla B\|_{L^2(A_R)}^2\right)^\frac{1}{1+2\mu}\left(\|\nabla u\|_{L^2(A_R)}^2+\|\nabla B\|_{L^2(A_R)}^2\right)^\frac{2\mu}{1+2\mu}\notag\\
\leq &C[f(2R)]^\frac{1}{1+2\mu}[RE'(R)]^\frac{2\mu}{1+2\mu}\\
\leq& C[R\ln RE'(R)]^\frac{2\mu}{1+2\mu}.\notag
\end{align}
Combining \eqref{ine4.17}, \eqref{ine4.19}, \eqref{ine4.20}, \eqref{ine4.33} and \eqref{ine4.35}, we obtain
\begin{align}\label{ine4.36}
\aligned
E(R)\leq& C[R\ln RE'(R)]^\frac{2\mu}{1+2\mu}+C[R\ln R E'(R)]^\frac{9-3p}{6-p}\\
&+CR^{\gamma_1}(\ln R)^{3\lambda}+CR^{\gamma_2}(\ln R)^{\lambda+2\mu}.
\endaligned
\end{align}
It is easy to verify that
$$\frac{2\mu}{1+2\mu}>\frac{9-3p}{6-p},\text{ if }2\mu>\frac{9-3p}{2p-3}.$$
By the Young inequality, we obtain
\begin{align*}
\aligned
C[R\ln R E'(R)]^\frac{9-3p}{6-p}&\leq \frac{1}{4}E(R_0)+C[R\ln R E'(R)]^\frac{2\mu}{1+2\mu}\\
&\leq\frac{1}{4}E(R)+C[R\ln R E'(R)]^\frac{2\mu}{1+2\mu},
\endaligned
\end{align*}
which with \eqref{ine4.36} imply
\begin{align*}
E(R)\leq C[R\ln R E'(R)]^\frac{2\mu}{1+2\mu}+CR^{\gamma_1}(\ln R)^{3\lambda}+CR^{\gamma_2}(\ln R)^{\lambda+2\mu}.
\end{align*}
Routinely, we obtain $u=B=0$.

\textbf{Assume that (C3) holds with $\beta\in\left[0,\frac{3}{q}-\frac{1}{2}\right)$.} In view of (C3), there exists a constant $R_1>3$ such that the following two inequalities hold for any $R>R_1$:
\begin{equation*}
\|u\|_{L^p\left(A_R\right)}\leq CR^\alpha,\;\|B\|_{L^q\left(A_R\right)}\leq CR^\beta(\ln R)^\mu.
\end{equation*}
Owing to \eqref{ine3.15}, it holds that
\begin{align*}
f(R)
\leq& CR^{1-\frac{6}{p}}\|u\|_{L^{p}(A_R)}^{2}+CR^{1-\frac{6}{q}}\|B\|_{L^{q}(A_R)}^{2}\notag\\
&+CR^{2-\frac{9}{p}}\|u\|_{L^{p}(A_R)}^{3}+CR^{-\frac{(6-q)p'}{(3-p')q}}\|u\|_{L^p(A_R)}^\frac{(6-q)p'}{(3-p')q}\|B\|_{L^q(A_R)}^2\notag\\
\leq&CR^{1-\frac{6}{p}+2\alpha}+CR^{1-\frac{6}{q}+2\beta}(\ln R)^{2\mu}+CR^{2-\frac{9}{p}+3\alpha}+CR^{-\frac{(6-q)p'}{(3-p')q}+\frac{(6-q)p'}{(3-p')q}\alpha+2\beta}(\ln R)^{2\mu}.\notag
\end{align*}
Consequently, we have
\begin{equation}\label{ine4.39}
f(R)\leq C(\ln R)^{2\mu}\text{ for any $R\geq R_2$,}
\end{equation}
where $R_2(>R_1)$ is a sufficiently large constant.
Using the H\"{o}lder inequality and \eqref{ine4.7}, we obtain
\begin{align}\label{ine4.40}
\aligned
CR^{-1}\|v\|_{L^3(A_R)}^3&\leq CR^{2-\frac{9}{p}}\|v\|_{L^p(A_R)}^3\leq CR^{2-\frac{9}{p}}\|u\|_{L^p(A_R)}^3.
\endaligned
\end{align}
Combining \eqref{ine4.17}, \eqref{ine4.18}, \eqref{ine4.40} and \eqref{ine4.20}, we obtain
\begin{align}\label{ine4.41}
E(R)\leq &C\left(\|\nabla u\|_{L^2(A_R)}^2+\|\nabla B\|_{L^2(A_R)}^2\right)+CR^{-1}\|u\|_{L^{p}(A_R)}\|B\|_{L^q(A_R)}^\frac{(4p-6)q}{(6-q)p}\|\nabla B\|_{L^2(A_R)}^\frac{12p-6q(p-1)}{(6-q)p}\notag\\
&+CR^{2-\frac{9}{p}}\|u\|_{L^p(A_R)}^3+CR^{2-\frac{3}{p}-\frac{6}{q}}\|u\|_{L^p(A_R)}\|B\|_{L^q(A_R)}^2\notag\\
\leq &C\left(\|\nabla u\|_{L^2(A_R)}^2+\|\nabla B\|_{L^2(A_R)}^2\right)+C[R\ln R E'(R)]^\frac{6p-3q(p-1)}{(6-q)p}\\
&+CR^{\gamma_1}X^3_{p,\alpha}(R)+CR^{\gamma_2}(\ln R)^{2\mu}.\notag
\end{align}
Denote
\begin{equation}\label{theta}
\theta=\max\left\{\frac{2\mu}{1+2\mu},\;\frac{6p-3q(p-1)}{(6-q)p}\right\}.
\end{equation}
It is not difficult to check that
$$\theta\in(0,1)\text{ and } 2\mu(1-\theta)\leq \theta.$$
Using \eqref{ine4.1} and \eqref{ine4.39}, we have
\begin{align}\label{ine4.42}
&\|\nabla u\|_{L^2(A_R)}^2+\|\nabla B\|_{L^2(A_R)}^2\notag\\
=&\left(\|\nabla u\|_{L^2(A_R)}^2+\|\nabla B\|_{L^2(A_R)}^2\right)^{1-\theta}\left(\|\nabla u\|_{L^2(A_R)}^2+\|\nabla B\|_{L^2(A_R)}^2\right)^\theta\notag\\
\leq &C[f(2R)]^{1-\theta}[RE'(R)]^\theta\\
\leq &C(\ln R)^{2\mu(1-\theta)}[RE'(R)]^\theta\notag\\
\leq& C[R\ln RE'(R)]^\theta.\notag
\end{align}
By the Young inequality, we obtain
\begin{align}\label{ine4.43}
\aligned
C[R\ln R E'(R)]^\frac{6p-3q(p-1)}{(6-q)p}&\leq \frac{1}{4}E(R_0)+C[R\ln R E'(R)]^\theta\\
&\leq\frac{1}{4}E(R)+C[R\ln R E'(R)]^\theta.
\endaligned
\end{align}
Since $\lim\limits_{R\rightarrow+\infty}[CR^{\gamma_1}X^3_{p,\alpha}(R)+CR^{\gamma_2}(\ln R)^{2\mu}]=0$, we can find a constant $R_3>R_0$ large enough such that
\begin{align}\label{ine4.44}
CR^{\gamma_1}X^3_{p,\alpha}(R)+CR^{\gamma_2}(\ln R)^{2\mu}\leq \frac{1}{4}E(R_0)\leq \frac{1}{4}E(R)\text{ for any $R\geq R_3$}.
\end{align}
Combining \eqref{ine4.41}, \eqref{ine4.42}, \eqref{ine4.43} and \eqref{ine4.44}, we conclude that
\begin{align*}
E(R)\leq C[R\ln R E'(R)]^\theta \text{ for any $R\geq R_3$.}
\end{align*}
Routinely, we obtain $u=B=0$.

\textbf{Assume that (C3) holds with $\beta=\frac{3}{q}-\frac{1}{2}$.} Combining \eqref{ine4.17}, \eqref{ine4.18}, \eqref{ine4.40} and \eqref{ine4.30}, we obtain
\begin{align*}
E(R)\leq &C\left(\|\nabla u\|_{L^2(A_R)}^2+\|\nabla B\|_{L^2(A_R)}^2\right)+CR^{-1}\|u\|_{L^{p}(A_R)}\|B\|_{L^q(A_R)}^\frac{(4p-6)q}{(6-q)p}\|\nabla B\|_{L^2(A_R)}^\frac{12p-6q(p-1)}{(6-q)p}\notag\\
&+CR^{2-\frac{9}{p}}\|u\|_{L^p(A_R)}^3+CR^{\frac{1}{2}-\frac{3}{2p}-\frac{3}{q}}\|u\|_{L^{p}(A_R)}\|B\|_{L^q(A_R)}^\frac{6p+pq-3q}{(6-q)p}\|\nabla B\|_{L^2(A_R)}^\frac{6p-3q(p-1)}{(6-q)p}\notag\\
\leq &C\left(\|\nabla u\|_{L^2(A_R)}^2+\|\nabla B\|_{L^2(A_R)}^2\right)+C[R\ln R E'(R)]^\frac{6p-3q(p-1)}{(6-q)p}\notag\\
&+CR^{\gamma_1}X^3_{p,\alpha}(R)+C[R\ln RE'(R)]^\frac{6p-3q(p-1)}{2(6-q)p}.\notag
\end{align*}
Routinely, we can derive that
\begin{align*}
E(R)\leq C[R\ln R E'(R)]^\theta\text{ for any large } R,
\end{align*}
where $\theta$ is given by \eqref{theta}.
Consequently, we obtain $u=B=0$.
\end{proof}

\begin{proof}[{\bf Proof of Theorem \ref{main4}}]
For the Hall-MHD equations, we take $\kappa=1$. Observing that
$$
R^{-\frac{1}{2}}\|\nabla B\|_{L^2(A_R)}+R^{-\frac{3}{q}}\|B\|_{L^q(A_R)}\leq R^{-\frac{1}{2}}[f(2R)]^\frac{1}{2}+CR^{-\frac{3}{q}+\beta}(\ln R)^\mu\leq C,
$$
for any $R$ large enough, we find that the following new term
$$
C\kappa\left(R^{-\frac{1}{2}}\|\nabla B\|_{L^2(A_R)}+R^{-\frac{3}{q}}\|B\|_{L^q(A_R)}\right)\|\nabla B\|_{L^2(A_R)}^2,
$$
which is produced by the Hall-MHD equations, can be absorbed by $C\|\nabla B\|_{L^2(A_R)}^2$. Hence, the rest of proof is almost the same with the MHD equations.
\end{proof}

\begin{proof}[{\bf Proof of Remark \ref{Rem1.6}}]
We only prove the case of the MHD equations.  According to the assumption (C4), we revise \eqref{ine4.34} slightly, and then we obtain
\begin{align}\label{ine4.46}
f(R)\leq C(\ln R)^{\lambda+2\mu}.
\end{align}
Denote
\begin{equation*}
\theta=\max\left\{\frac{2\lambda+2\mu}{1+\lambda+2\mu},\;\frac{9-3p}{6-p},\;\frac{1}{2}\right\}.
\end{equation*}
Then it holds that
$$\theta\in(0,1)\text{ and }\lambda+(\lambda+2\mu)(1-\theta)\leq\theta.$$
Using \eqref{ine4.1} and \eqref{ine4.46}, we have
\begin{align}\label{ine4.47}
&\|\nabla u\|_{L^2(A_R)}^2+\|\nabla B\|_{L^2(A_R)}^2\notag\\
=&\left(\|\nabla u\|_{L^2(A_R)}^2+\|\nabla B\|_{L^2(A_R)}^2\right)^{1-\theta}\left(\|\nabla u\|_{L^2(A_R)}^2+\|\nabla B\|_{L^2(A_R)}^2\right)^\theta\notag\\
\leq &C[f(2R)]^{1-\theta}[RE'(R)]^\theta\\
\leq &C(\ln R)^{(\lambda+2\mu)(1-\theta)}[RE'(R)]^\theta\notag\\
\leq& C[R\ln RE'(R)]^\theta.\notag
\end{align}
Using the H\"{o}lder inequality, the Sobolev-Poincar\'{e} inequality, \eqref{ine4.1} and \eqref{ine4.46}, we have
\begin{align}\label{ine4.48}
&CR^{-1}\|u\|_{L^{p}(A_R)}\|H\|_{L^{2p'}(A_R)}^2\notag\\
\leq&CR^{1-\frac{3}{p}}\|u\|_{L^{p}(A_R)}\|H\|_{L^6(A_R)}^2\notag\\
\leq& CR^{1-\frac{3}{p}}\|u\|_{L^{p}(A_R)}\|\nabla B\|_{L^2(A_R)}^2\notag\\
=&CR^{1-\frac{3}{p}}\|u\|_{L^{p}(A_R)}\|\nabla B\|_{L^2(A_R)}^{2(1-\theta)}\|\nabla B\|_{L^2(A_R)}^{2\theta}\\
\leq& CR^{1-\frac{3}{p}}\|u\|_{L^{p}(A_R)} [f(2R)]^{1-\theta}[RE'(R)]^\theta\notag\\
\leq& C(\ln R)^{\lambda+(\lambda+2\mu)(1-\theta)}[RE'(R)]^\theta\notag\\
\leq& C[R\ln RE'(R)]^\theta.\notag
\end{align}
By the H\"{o}lder inequality, the Sobolev-Poincar\'{e} inequality and \eqref{ine4.1}, we obtain
\begin{align}\label{ine4.49}
&CR^{\frac{1}{2}-\frac{3}{2p}-\frac{3}{q}}\|u\|_{L^{p}(A_R)}\|H\|_{L^{2p'}(A_R)}\|B\|_{L^q(A_R)}\notag\\
\leq&CR^{\frac{3}{2}-\frac{3}{p}-\frac{3}{q}}\|u\|_{L^{p}(A_R)}\|H\|_{L^6(A_R)}\|B\|_{L^q(A_R)}\notag\\
\leq&CR^{\frac{3}{2}-\frac{3}{p}-\frac{3}{q}}\|u\|_{L^{p}(A_R)}\|\nabla B\|_{L^2(A_R)}\|B\|_{L^q(A_R)}\\
\leq&C(\ln R)^{\lambda+\mu}[RE'(R)]^\frac{1}{2}\notag\\
\leq&C[R\ln RE'(R)]^\frac{1}{2}.\notag
\end{align}
Combining \eqref{ine4.17}, \eqref{ine4.19}, \eqref{ine4.47}, \eqref{ine4.48} and \eqref{ine4.49}, we deduce that
$$
E(R)\leq C[R\ln RE'(R)]^\theta+C[R\ln RE'(R)]^\frac{9-3p}{6-p}+C[R\ln RE'(R)]^\frac{1}{2}+CR^{-1}(\ln R)^{3\lambda}.
$$
Routinely, we can derive that
\begin{align*}
E(R)\leq C[R\ln R E'(R)]^\theta\text{ for any large } R.
\end{align*}
Consequently, we get $u=B=0$.
\end{proof}

\subsection*{Acknowledgements.}
This work was supported by Science Foundation for the Excellent Youth Scholars of Higher Education of Anhui Province (Grant No. 2023AH030073), and Domestic Study and Research Support Program for Young Key Teachers of Higher Education of Anhui Province (Grant No. JNFX2025027).


\subsection*{Data Availability Statement}
No data was used for the research described in the article.


\end {document}